\documentclass{amsart}
\usepackage{amssymb,amsmath,amsthm,epsf,epsfig,dsfont,bbm}
\usepackage{latexsym}
\usepackage{upref, eucal}
\usepackage[all]{xy}
\usepackage[margin=100pt]{geometry}

\newcommand {\nc} {\newcommand}
\newcommand {\enm} {\ensuremath}

\def \d{\delta}

\nc {\bdm} {\begin{displaymath}}
\nc {\edm} {\end{displaymath}}

\newtheorem {theorem} {\bf{Theorem}}[section]
\newtheorem {lemma}[theorem] {\bf Lemma}
\newtheorem {proposition}[theorem] {\bf Proposition}

\newtheorem {definition}[theorem]{\bf Definition}
\newtheorem {corollary}[theorem] {\bf Corollary}
\numberwithin {equation}{section}

\newcommand\FF{\mathbb{F}}

\newcommand{\Ou}{\enm{\mathcal{O}}}

\nc{\J}{\enm{\mathcal{J} }}
\nc {\Z} {\enm{\mathbb{Z}}}
\nc {\form}[1] {\enm{\mbox{\underline{for}}}_{#1}}
\nc {\prol}[1] {\enm{\mbox{\underline{prol}}_{{#1}^*}}}

\nc {\stk} {\stackrel}

\newcommand{\map}{\rightarrow}

\newcommand{\Pn}[2] {\ensuremath{ {\mathbb{P}}^{#1}_{#2}}}
\nc{\Quot}[3]{\enm{ {\mathfrak{Quot}_{ {#1}/{#2}/{#3}}}}}
\nc{\Hilb}[2]{\enm{ {\mathfrak{Hilb}_{ {#1}/{#2}}}}}
\newcommand{\mfrak}[1]{\mathfrak{#1}}

\newcommand{\bb}[1]{\mathbb{#1}}

\nc {\Coh}[4] {\ensuremath{H^{#1}(\Pn{#2}{},{#3}({#4}))}}
\nc {\Ch}[3] {\enm{H^{#1}(X_t,{#2}_t({#3}))}}
\nc {\Qphi}[4]{\enm{ {\mathfrak{Quot}^{~#4}_{ {#1}/{#2}/{#3}}}}}
\nc {\Gra}[4]{\enm{ {\mathfrak{Grass}_{#2}({#3},{#4})}}}
\nc {\HomA}[2]{\enm{\mathrm{Hom}_A{#1}{#2}}}
\nc {\tr}{\mathrm{tr}}

\nc {\C}[2]{\enm{\left(\begin{array}{l} {#1} \\ {#2} \end{array} \right)}}
\nc {\mat}[4]{\enm{\left(\begin{array}{ll}{#1} & {#2} \\ {#3} & {#4}
\end{array}\right)}}

\def \vp{\varphi}
\def \mb{\mbox}

 \def \Z{{\mathbb Z}}

   \def \h{\hat{\ }}

\def \d{\delta}   \def \bF{{\bf F}}
   
\def \bA{{\mathbb A}}  

 \def \bF{{\mathbb F}}

\def \hG{\hat{\mathbb{G}}_{\mathrm{a}}}

\def \R1{R((q))[q']\h}

\usepackage{xcolor}

\DeclareMathOperator{\Spf}{\mathrm{Spf}}

\DeclareMathOperator{\Lie}{\mathrm{Lie}}

\newcommand{\Hom}{\mathrm{Hom}}

\newcommand{\Ext}{\mathrm{Ext}}

\newcommand{\bI}{{\bf I}}

\newcommand{\teich}{\mathfrak{v}}

\nc{\tR}{\tilde{R}}
\nc{\bx}{\mathbf{x}}
\nc{\by}{\mathbf{y}}
\nc{\bz}{\mathbf{z}}
\nc{\ba}{\mathbf{a}}
\nc{\Fp}{\tilde{F}}
\nc{\Rp}{\tilde{R}}
\nc{\mlow}{m_{\mathrm{l}}}
\nc{\mup}{m_{\mathrm{u}}}
\nc{\ord}{{\mathrm{ord}}}
\nc{\bX}{\mathbf{X}}
\nc{\bH}{\mathbf{H}}
\nc{\bXp}{\mathbf{X}_{\mathrm{prim}}}
\nc{\bPsi}{\mathbf{\Psi}}
\nc{\mult}{\mathrm{mult}}
\nc{\mbB}{\mathbbm{B}}
\nc{\mfor}[1]{{#1}^{\mathrm{for}}}
\nc{\Hdr}{\bH^{*}_{\mathrm{dR}}}
\nc{\bt}{{\bf t}}
\nc{\beqar}{\begin{eqnarray*}}
\nc{\eeqar}{\end{eqnarray*}}
\nc{\fra}{\mfrak{f}}
\nc{\tW}{\tilde{W}}

\nc{\bo}{{\bf b}}
\nc{\hq}{{\hat{q}}}
\nc{\wHH}{\mathbf{H}_{\d}(E)}
\nc{\qH}{\mfrak{q}_{\wHH}}
\nc{\bfH}{\mathbf{H}(E)}
\nc{\admissible}{elliptic }
\nc {\rnk}{\mathrm{rank}}

\nc {\cblue}{\color{black}}

\title
[Delta Theory of Anderson Modules I: Differential Characters]
{Delta Theory of Anderson Modules I: Differential Characters}

\author{Sudip Pandit}
\address{Department of Mathematics, 
Indian Institute of Technology Gandhinagar, Gujarat 382355}
\email{sudip.pandit@iitgn.ac.in}

\author{Arnab Saha}
\address{Department of Mathematics, 
Indian Institute of Technology Gandhinagar, Gujarat 382355}
\email{arnab.saha@iitgn.ac.in}

\date{}

\date{}

\subjclass[2010]{Primary 11G09, 14G17, 14L05, 14L15, 14L17, 14B20.}

\keywords{Witt vectors, jet spaces, Drinfeld module, Anderson module, $\d$-character}

\allowdisplaybreaks

\begin{document}
\begin{abstract}
In this article we develop the theory of differential or delta characters (the 
arithmetic analogue of Manin characters) of Anderson modules. 
Here we generalize the construction by Borger and Saha of the canonical 
finite rank $R$-module $\mathbf{H}(E)$ with a semilinear operator on it
 to any Anderson module $E$,
where $R$ is the base ring which is a $\pi$-adically complete discrete valuation ring with a fixed lift of Frobenius $\phi$ on it. 
Then we show that $\mathbf{H}(E)$ admits a functorial map to the de Rham 
cohomology $\mathbf{H}_{\mathrm{dR}}^*(E)$ of $E$ which also preserves the
Hodge filtration. 

We also prove that the module of delta characters 
$\mathbf{X}_{\infty}(E)$ is  finite and free as an $R\{\phi^{*}\}$-module. 
This leads to a strengthened version of an analogous result by Buium on the generation of differential characters of abelian varieties. 
We also construct a family of differential modular functions that play the 
analogous role of $f_{\mathrm{jet}}$ constructed by Buium for elliptic curves.

In a subsequent article, the finite rank $R$-module $\mathbf{H}(E)$ will lead to 
the construction of a canonical $z$-isocrystal $\mathbf{H}_{\delta}(E)$ with a Hodge-Pink
filtration on it and we will show that $\mathbf{H}_{\delta}(E)$ is an admissible
$z$-isocrystal. 
\end{abstract}

\maketitle
\section{Introduction}

 The theory of arithmetic jet spaces, developed by Buium, is inspired by 
the differential algebra theory over function fields. In differential algebra,
let $K$ be a function field
with a fixed derivation $\partial$ on it. For any $K$-algebra $B$ we denote 
${D}_n(B)$ the ring of truncated polynomials of degree $n$ with 
coefficients in $B$. Then ${D}_n(B)$ naturally inherits a $K$-algebra 
structure by the Hasse-Schmidt differentiation map \cite{gil02}. 
When the fixed derivation $\partial$
on $K$ is non-zero, the $K$-algebra structure on ${D}_n(B)$ described 
as above is not the usual scalar multiplication.
Given a (group) scheme $G$ defined over $K$ with fixed 
derivation $\partial$ on it, one defines the $n$-th jet space functor $J^nG$ 
whose
$B$-points are $G({D}_n(B))$. Then the functor 
$J^nG$ is representable by a (group)
scheme and $\{J^nG\}_{n=0}^\infty$ form a system of schemes which is known
as the {\it canonical prolongation sequence} associated to $G$. In fact, if the
fixed derivation $\partial$ on $R$ is zero, then $J^1G$ is the tangent 
bundle of $G$.

The global sections $\Ou(J^nG)$ can be thought of as the ring of the $n$-th
order differential functions of $G$. For instance, if 
$G$ is an elliptic curve,
then one can show that there always exists a differential function $\Theta 
\in \Ou(J^2G)$ which is also a morphism of group schemes from $J^2G$ to the
additive group scheme $\bb{G}_a$. Such a $\Theta$ is called a differential
character and is also known as Manin character. As for example, if $G$ is given 
by the Legendre equation $y^2=x(x-1)(x-t)$ defined over the function field $K=
\bb{C}(t)$ with derivation $\partial = \frac{d}{dt}$ on it, then 
$$ 
        \Theta(x,y,x',y',x'',y'')= \frac{y}{2(x-t)^2} - \frac{d}{dt}
        \left[2t(t-1)\frac{x'}{y} \right] + 2t(t-1)x'\frac{y'}{y^2}
$$
where $x,y,x',y',x'',y''$ are local coordinates of $J^2G$. The existence of 
such a $\Theta$ is due to the Picard-Fuchs equation. Given a $K$-rational
point $P \in G(K)$, using the derivation $\partial$, we can consider its
canonical lift $\nabla :G(K) \map J^2G(K)$ given by $\nabla(x,y) =
(x,y,\partial x, \partial y, \partial^2 x, \partial^2 y)$. However 
$\nabla$ is only a map at the level of $K$-points and not of schemes.
The composition $\Theta \circ \nabla: G(K) \map \bb{G}_a(K)$ is then 
a group homomorphism, and hence the torsion points of $G(K)$ are in the 
kernel of the map since $\bb{G}_a(K)$ is torsion free. Such a character $\Theta$
was applied by Manin to prove the Lang-Mordell conjecture for abelian varieties 
over function fields \cite{Manin63}. Later, Buium gave an independent proof 
by using other methods but still using the Manin map \cite{bui92}.

The theory of arithmetic jet spaces over number rings, as 
developed by Buium, also respects similar
lines. Here the notion of derivation on the base ring is replaced by a 
$\pi$-derivation $\d$ which naturally arises from the $\pi$-typical Witt 
vectors. The analogy of $\pi$-derivations with the usual derivations is as 
follows:
note that giving a derivation on a ring $B$ is equivalent to giving a ring 
homomorphism $B \map {D}_1(B)$. In the same essence, giving a $\pi$-derivation
on a ring $B$ is equivalent to giving a ring map $B \map W_1(B)$ where $W_1(B)$
is the ring of $\pi$-typical Witt vectors of length 2. As for example,
if $R$ is an unramified $p$-adically completed 
extension of $\Z_p$ and 
$\pi=p$, then the Fermat quotient operator $\d x = \frac{\phi(x)- x^p}{p}$ 
is the unique $p$-derivation, where the endomorphism $\phi:R \map R$ is the 
lift of the $p$-th power Frobenius endomorphism of $R/pR$. Such a $\phi$ is the
unique lift of Frobenius associated to the $\pi$-derivation $\d$.

Buium developed the theory of $\d$-characters for 
abelian varieties in \cite{Buium 1995}. 
This led to the development of delta geometry 
\cite{Barc,Buium 2000,Buium-Miller,BuSa1, BuSa2,Hurl} and 
 found remarkable applications in 
diophantine geometry as in \cite{Buium 1996, BP}. Such $\d$-geometric
objects have also been used in the construction of the prismatic cohomology by
Bhatt and Scholze \cite{BhSc}.

In \cite{Drinfeld_74}, Drinfeld introduced elliptic modules 
(now known as Drinfeld modules), and using them he 
proved the global Langlands' correspondence for $\mathrm{GL}_2$ over function 
fields in \cite{Drinfeld_80}. 
Later in  \cite{Anderson 1986}, Anderson extended these objects to higher 
dimensions, called $t$-modules (now known as Anderson modules). 

In a previous article by Borger and Saha \cite{BS b}, 
for Drinfeld modules $E$ 
defined over a discrete valuation ring $R$ with a fixed $\pi$-derivation $\d$, 
they construct a canonical $R$-module $\bH(E)$ with a semilinear operator
$\fra$ on it along with a filtration. It was further shown 
that $\bH(E)$ is in fact a free $R$-module of 
finite rank. This implied that the $R\{\phi^*\}$-module $\bX_\infty(E)$ (the
module of $\d$-characters) is free and finitely generated.

In this paper, we generalize the construction of $\bH(E)$ to 
abelian Anderson modules (defined in Section
\ref{AndMod}). 
In Theorem \ref{finite_dim_thm}, we show that 
$\bH(E)$ is a free and
finitely generated $R$-module. We would like to remark that our proof of
showing the finite generation of $\bH(E)$ is distinct from that of \cite{BS b}
as the latter method did not have an immediate generalization to the case of 
Anderson modules.
As a consequence we obtain our Corollary \ref{FinGenX} which shows that the 
$R\{\phi^*\}$-module $\bX_\infty(E)$ is free and finitely generated. 
Our method also extends to the case of abelian schemes which would imply
a strengthened version of Theorem A in \cite{Buium 1995} by Buium. We also show that $\bH(E)$ functorially maps to the de Rham cohomology $\Hdr(E)$ of $E$.

We now explain our results in greater detail. 

\cblue
Let $q=p^{h}$ be a fixed prime power and $\mathcal{C}$ be a projective, geometrically connected, smooth curve over $\mathbb{F}_{q}$. Fix an $\mathbb{F}_{q}$-rational point $\infty$ on $\mathcal{C}$. Let $A$ be the ring of functions regular outside $\infty$. Fix a maximal ideal $\mathfrak{p}$ of the Dedekind domain
$A$ and $z\in \mathfrak{p}\setminus \mathfrak{p}^{2}$ be an element of degree $f$ such that $A$ is a separable extension of $\FF_{q}[z]$. Note that it is always possible to find such uniformizer for any $A$ (cf. page 18, Proposition $1.4$ in \cite{Silverman}). Let $\hat{A}$  be the $\mathfrak{p}$-adic completion  
of $A$ and $\pi$ be the image of $z$ in $\hat{A}$. Then $\pi$ generates the maximal ideal $\hat{\mathfrak{p}}$. Let $k$ denote the residue field $A/\mathfrak{p}$ which is finite and let $\hq=q^{f}$ be its cardinality. Note that the quotient map $\hat{A}\longrightarrow \FF_{\hq}$ has a unique section and thus $\hat{A}$ is not only an $\mathbb{F}_{q}$-algebra but also an $\FF_{\hq}$-algebra.
Let $\hat{A}$  be the $\mathfrak{p}$-adic completion  
of $A$ and $\pi$ be the image of $z$ in $\hat{A}$. Fix a flat $\hat{A}$-algebra $R$ which is 
$\mathfrak{p}$-adically complete discrete valuation ring. 

\color{black}

Let $R$ have a lift of Frobenius which is an $\hat{A}$-algebra
endomorphism $\phi:R\map R$ that satisfies $\phi(x) \equiv x^{\hat{q}}
\bmod (\mfrak{p}R)$. 
Associated to such a $\phi$, consider the unique $\pi$-derivation $\d:R \map
R$ given by $\d x = \frac{\phi(x)- x^{\hq}}{\pi}$. Then the $\pi$-derivation
$\d$ satisfies the following:
\begin{enumerate}
\item $\d(x+y) = \d x + \d y$
\item $\d(xy) = x^{\hq}\d y + y^{\hq} \d x+ \pi \d x \d y$
\end{enumerate}
for all $x,y \in R$.
Note that $\d$ is uniquely characterised by a ring map of $R$ to the Witt 
vectors of length $2$. Also note that since $R$ is of prime characteristic, 
the operator $\d$ is additive.
For an $R$-algebra $B$, let $W_n(B)$ denote the ring of Witt vectors of
length $n+1$ as in \cite{bor11b}. Then note that $W_n(B)$ is also naturally 
an $R$-algebra via the universal property of Witt vectors.
Given 
any $\pi$-formal scheme $X$ over $\Spf R$ the $n$-th jet space $J^nX$ is 
functorially defined as 
\begin{align}
\label{jetfuncdef}
J^nX(B) := X(W_n(B))
\end{align}
for all $\pi$-adically complete $R$-algebra $B$. Then $J^nX$ is representable 
by a $\pi$-formal scheme over $\Spf R$ that is obtained by glueing over open
affine covers of $X$- this is analogous to Buium's construction in the 
mixed characteristic case as in \cite{Buium 1995}.
For more details on the representability of the jet functor, the reader may
see \cite{bps} \cite{bor11b}. 

\cblue
We now discuss the specific implications of the above definition of jet 
spaces to our setting
where $E$ is an admissible $A$-module of dimension $d$ over $\Spf R$
(defined after Definition \ref{abAnddef} in Section $2$). 
Firstly we will now review the group scheme structure of the jet scheme 
associated to $E$.
Since $E$ is an admissible $A$-module of dimension $d$, the underlying 
$\pi$-formal group scheme $E$ is isomorphic to $\hG^d$ where $\hG$ is 
the additive $\pi$-formal group scheme over $\Spf R$. Hence by the definition
of the jet space functor in (\ref{jetfuncdef}) we have an isomorphism 
$J^nE \simeq W_n^d$ 
where $W_n$ is the affine $(n+1)$-space $\hat{\bA}^{n+1}$ with the 
group structure of the additive $\pi$-typical Witt vectors of length $n+1$.

Let us now examine the natural extension of the $A$-module structure on $E$
to that of $J^nE$. Let the $A$-module
structure on $E$ be denoted by $\vp_E:A \map \mathrm{End}_R(E)$. Hence for any
$a \in A$ and a $\pi$-adically complete $R$-algebra $C$ we have the functorial
map $\vp_E(a): E(C) \map E(C)$ that is a group homomorphism.

Define the $A$-module structure on $J^nE$, denoted by $\vp_{J^nE}$, as
$\vp_{J^nE}(a)(x):= \vp_E(x)$ for all $x \in E(W_n(B)) = J^nE(B)$.
Hence this gives the natural $A$-module structure on $J^nE$ that
is induced from the $A$-module structure on $E$.


\color{black}
Hence we have the following short exact sequence of $A$-module schemes
\begin{align}
\label{fshort}
0 \map N^n \stk{i}{\map} J^nE \stk{u}{\map} E \map 0 
\end{align}
where $u:J^nE \map E$ is the natural projection map and $N^n$ is the
kernel of $u$. We would like to remark that even if $E$ is an Anderson module,
$J^nE$ is not one.

Let $E$ be an abelian Anderson $A$-module of dimension $d$ and rank $r$ over $\Spf R$, and the $A$-linear action of $z$ is given by  $$\varphi_{E}(z)=\displaystyle{\sum_{i=0}^{s}}A_{i}{\tau}^{i}$$
where $A_{0}=\pi I+V(z)$ and $V$ is a nilpotent matrix. We define 
\begin{align*}
	V^{\perp}:=\{u\in \mathrm{Mat}_{1\times d}(R)~| uV(z)=0\}
\end{align*}
Note that $V^{\perp}$ is a free $R$-submodule of $R^{d}$ with rank equal to $h=d-\mathrm{rank}(V(z))$. 
Let $\hG$ denote the $\pi$-formal additive group scheme over $\Spf R$. 
Note that $\hG$ 
naturally has an $R$-module structure given by scalar multiplication.
Then let $(\hG,\vp_{\hG})$ denote the $\pi$-formal additive group scheme 
$\hG$ with the tautological $A$-action that is 
induced from the characteristic map $\theta:A \map R$. 
An $A$-linear morphism of group schemes from $J^{n}E$ to $\hat{\mathbb{G}}_{a}$ is called a \textit{differential} or  \textit{delta character of order $\leq n$} of $E$. We will denote the group of all $\d$-characters of order $n$ as
 ${\bX}_{n}(E)$.

Since $\hG$ is an $R$-module $\pi$-formal scheme over $S$, $\bX_n(E)$
naturally becomes an $R$-module.
Also the inverse system $J^{n+1}E\xrightarrow{u} J^{n}E$ defines a directed system $$\ldots \xrightarrow{u^{*}} {\bX}_{n}(E)\xrightarrow{u^{*}} {\bX}_{n+1}(E) \xrightarrow{u^{*}}\ldots$$ 
of $R$-modules via the pullback map. We define 
$${\bX}_{\infty}(E)=\varinjlim {\bX}_{n}(E).$$

Then the Frobenius morphism on the jet spaces $\phi:J^{n+1}E \map J^nE$ endows 
$\bX_\infty(E)$ with an $R\{\phi^*\}$-module structure where $\phi^*$ is 
additive and $R$-semilinear, that is, $\phi^*$ satisfies 
$$
\phi^*(r \Theta ) = \phi(r) \phi^*\Theta
$$
and $\phi^*\Theta$ is the relative pullback of $\Theta$ induced from $\phi$.


In \cite{BS b}, \cite{BS a}, for all $n$ the lateral Frobenius $\mfrak{f}:
N^n \map N^{n-1}$ was constructed which makes $\{N^n\}_{n=1}^\infty$
into a prolongation sequence. Then $\fra$ satisfies 
\begin{align}
\label{comp}
\phi \circ i \circ \fra = \phi^{\circ 2} \circ i.
\end{align}

Hence pulling back via $\mfrak{f}$ makes
$\varinjlim \Hom_A(N^n,\hG)$ into an $R\{\mfrak{f}^*\}$-module.  
Along the principles of \cite{BS b}, we define the $R$-module
	$$
	\bH(E) 	:= 
	 \varinjlim\frac{\Hom(N^n,\hG)}{i^*\phi^*(\bX_{n-1}(E)_{\phi})}.
	$$
Then by (\ref{comp}), $\mfrak{f}^*$ on $\varinjlim \Hom(N^n,\hG)$ 
descends to $\bH(E)$. Consider the $R$-submodule 
	$$
	\bXp(E):= \varinjlim \bX_n(E)/\phi^*(\bX_{n-1}(E)_{\phi})
	$$	
which satisfies 
\begin{align}
\label{Xpshort}
0 \map \bXp(E) \map \bH(E) \map \mathbf{I}(E) \map 0
\end{align}
where ${\mathbf{I}}(E)$ is an $R$-submodule of $\Ext_A(E,\hG)$ as defined in
Section \ref{fingenX}.
In Section $6$ we discuss the general theory of extensions of $A$-module 
schemes following Section $5$ of \cite{Gekeler_a}. Let 
$\bH^{*}_{\mathrm{dR}}(E)$ 
denote the first de Rham cohomology module of $E$ and $\Ext_A(E,\hG)$ be the 
$R$-module that classifies isomorphism classes of
 $A$-module $\pi$-formal schemes that 
are extensions of $E$ by $\hG$. Then our first result is:
\begin{theorem} \label{intro-ext-derham}
For any abelian Anderson $A$-module $E$ of rank $r$ and dimension $d$,
we have
\begin{enumerate}
\item The $R$-module $\bH^{*}_{\mathrm{dR}}(E)$ 
 is free of rank $r$.
\item The $R$-module $\Ext_A(E,\hG)$ is free of rank $r-h$. 
\end{enumerate}
\end{theorem}
The above results extend the previous result of Gekeler on Drinfeld modules in \cite{Gekeler_b} and of Brownawell and Papanikolas on $t$-modules over a field for $A=\FF_{q}[t]$  in \cite{BrPa}. Our results in Theorem \ref{intro-ext-derham} hold for a general $A$ and any abelian Anderson module $E$. 
 
In Section \ref{ziso} we obtain a natural 
map of short exact sequences

\begin{equation}
	\label{diag-crys-limit-intro}
	\xymatrix{
0  \ar[r] &{\bX}_{\mathrm{prim}}(E)\ar[d]^\Upsilon \ar[r] &{\bH}(E)\ar[d]^\Phi \ar[r] &\mathbf{I}(E)\ar@{^{(}->}[d] \ar[r]&0 \\
0 \ar[r] &\mathrm{Lie} (E)^{*} \ar[r] &\bH^{*}_{\mathrm{dR}}({E}) \ar[r] &\mathrm{Ext}_{A}(E,\hat{\mathbb{G}}_{a}) \ar[r] & 0.}
\end{equation}

Note that apriori it is not clear whether $\bXp(E)$ and $\bH(E)$ are finite 
and free $R$-modules. However it is easy to show that 
after tensoring with $K$ (the fraction field of 
$R$), they are finite dimensional $K$-vector spaces (c.f. Corollary  
\ref{rk_X_prim}). Our next theorem states
that the above $R$-modules $\bXp(E)$ and $\bH(E)$ are finite and 
free over $R$.

\begin{theorem} 
For any abelian Anderson $A$-module $E$ of rank $r$ and dimension $d$,
we have
\begin{enumerate}
\item The $R$-module $\bXp(E)$ is free of rank $h$.
\item The $R$-module ${\bH}(E)$ is free of rank $\leq r$. 
\end{enumerate}
\end{theorem}

As a consequence of the above finiteness result, we obtain the following in
Section \ref{fingenX}:

\begin{theorem}\label{fg_diff_char}
For any abelian Anderson $A$-module $E$ of rank $r$ and dimension $d$, 
${\bX}_{\infty}(E)$  is  freely generated as an
$R\{\phi^{*}\}$-module by $h$ $\d$-characters of order at most $r-h+1$.
\end{theorem}

Our method extends to the case of abelian schemes which would imply
a strengthened version of Theorem A of \cite{Buium 1995} by Buium. This is 
proved in \cite{PS-2}.
A Drinfeld module $E$ is said to be 
a \emph{Canonical Lift} (CL) if there is an $A$-linear endomorphism 
$\psi(x)\in \mathrm{Hom}_A(E,E)$ such that $\psi(x)
\equiv x^\hq ~\mathrm{mod}~\pi$ for all  $x \in \hat{\mathbb{G}}_{a}$. 
\begin{theorem}\label{CL_intro}
Let $E$ be a Drinfeld module of rank $r$ over $R$. There exist
 restricted power series $f_{i}$ in $2r+1$  variables over $R$ such that $E$ has CL if and only if
$$ f_{i}(a_{1},a_{2},\ldots ,a_{r},a_{1}',a_{2}',\ldots ,a_{r}',a_{r}^{-1})=0,\ \textit{for all}\ i=1, \dots , r-1$$ 
where $a_1,\dots, a_r$ are the structure coefficients of $E$.
\end{theorem}
The above implies that
the $\mathrm{CL}$ points form a `$\delta$-closed subset'
 inside the Drinfeld modular scheme for Drinfeld modules of rank $r$.
In particular, when $E$ is a Drinfeld module of rank $2$, the delta function 
$f_{1}$ is the analogue of the $\d$-modular form $f_{\mathrm{jet}}$ 
constructed in Section 4 of \cite{Buium 2000} by Buium. 
This will be addressed in detail in a subsequent paper.

{\bf Acknowledgements.} 
We wish to thank the anonymous referee for carefully reading our article and for
the suggestions which led to deeper clarifications and enrichment of this paper.
The paper has greatly benefitted from enlightening discussions that took 
place between Jim Borger and the second author.
The second author was partially supported by the SERB grant 
SRG/2020/002248.

\section{Anderson Modules }
\label{AndMod}


Let $R$ be a flat $\hat{A}$-algebra (that is $\pi$-torsion free) which is a 
$\mathfrak{p}$-adically complete discrete valuation ring. Hence we have the injection of rings
 $$ \theta :A\longrightarrow \hat{A} \longrightarrow R$$ 
and $\theta$ is called the characteristic map.
Let us also fix an $\hat{A}$-algebra endomorphism $\phi:R\longrightarrow R$ with $\phi(x)\equiv x^{\hq}\mod(\mathfrak{p}R)$ for $x\in R$. The affine formal scheme Spf $R$ is denoted by $S$. Let $K$ be the field of fraction of $R$ and $l$ be the residue field. Then by the Cohen structure theorem we have $R\simeq l[[z]]$ and $K\simeq l((z))$. 
Given any $R$-module $M$ we denote  the $K$-vector space $M\otimes _{R}K$  by 
$M_{K}$.

\begin{definition} An \emph{$A$-module scheme}  is a pair $(E,\varphi_{E})$ consisting of a smooth $\pi$-formal group scheme $E$ over $S$ and a ring homomorphism $$\varphi_{E}:A\longrightarrow \mathrm{End}_{R}(E).$$ 
\end{definition}

 Throughout this article every $A$-module scheme is considered over $S$ unless it is specified.
Given two  $A$-module schemes $G$ and $G^{'}$  we denote the set of all $\mathbb{F}_{q}$-linear morphisms of $S$-group schemes from $G$ to $G^{'}$ by $\mathrm{Hom}_{R,\mathbb{F}_{q}}(G,G^{'})$. \medskip

Let $E$ be an $A$-module scheme. We will denote ${\bX}(E):=\mathrm{Hom}_{A}(E,\hat{\mathbb{G}}_{a})$, the group of all $A$-linear $\pi$-formal group scheme 
morphisms from $E$ to $\hat{\mathbb{G}}_{a}$, where $\hat{\mathbb{G}}_{a}$ is the $A$-module scheme with $A$-action induced from the characteristic map.
For every positive integer $n$, we define the arithmetic jet space $J^{n}E$ of 
$E$ (as defined in the Introduction) by the functor of points as 
$$J^{n}E(B):=E(W_{n}(B))$$
for any $\pi$-adically complete $R$-algebra $B$. In fact $J^{n}E$ is representable by an $A$-module scheme over $S$. \cblue For a detailed discussion on the positive characteristic jet spaces we refer the reader to section $4$A in \cite{BS b}.  \color{black} 
Then we have the following short exact sequence of $A$-module schemes
\begin{align}
\label{fshort1}
0 \map N^n \stk{i}{\map} J^nE \stk{u}{\map} E \map 0 
\end{align}
where $u:J^nE \map E$ is the natural projection map and $N^n$ is the 
kernel of $u$. We would like to remark that even if $E$ is an Anderson module,
$J^nE$ is not one.

The set of all elements of ${\bX}(J^{n}E)$ are called the \emph{$\d$-character of order $n$} of $E$. For the sake of convenience we will 
denote ${\bX}(J^{n}E)$ as ${\bX}_{n}(E)$ for the rest of the article. Given any $A$-module scheme $E$ consider the group $M(E):=\mathrm{Hom}_{R,\mathbb{F}_{q}}(E, \hat{\mathbb{G}}_{a})$ of $\mathbb{F}_{q}$-linear maps from $E$ to $\hat{\mathbb{G}}_{a}$ over $S$. Then $M(E)$ has a natural left action by $A$ and a right action by $R$ given by  
$$(a\otimes b)m:=b\circ m\circ \varphi_{E}(a)\quad \mathrm{for \ any \ a\otimes b \in A\otimes_{\mathbb{F}_{q}}R},$$
and the above action is given by the following composition
$$ E \xrightarrow{\varphi_{E}(a)}E\xrightarrow{m} \hat{\mathbb{G}}_{a}\xrightarrow{b}\hat{\mathbb{G}}_{a}.$$ The $A\otimes_{\mathbb{F}_{q}}R$-module $M(E)$ is called the  $A$-\emph{motive} associated to $E$.  
\cblue
Given an $A$-module scheme $E$, note that it naturally endows an $A$-module
structure on $\Lie (E)$, denoted by $\vp_{\Lie (E)}$, and is given by the 
derivative map $D\vp_E(a)$ for all $a \in A$, that is 
$$
\vp_{\Lie (E)}(a)(y) := D\vp_E(a)(y)
$$
for all $a \in A$ and $y \in \Lie (E)$.

\color{black}

\begin{definition}  
\label{abAnddef}
An abelian Anderson $A$-module of rank $r$ and dimension $d$ over $S$ is an affine $A$-module scheme  $(E,\varphi_{E})$ over  $S$  of relative dimension $d$ such that 
\begin{itemize}

\item[(i)] $E$ is isomorphic to $\hat{\mathbb{G}}_{a}^{d}$ as $\mathbb{F}_{q}$-modules
\item[(ii)] $(\varphi_{\mathrm{Lie}E}(a)-\theta(a))^{d}=0$, for all $a \in A$
\item[(iii)] The associated $A$-motive $M(E)$ is a locally free $A\otimes_{\mathbb{F}_{q}}R$ module of rank $r$.
\end{itemize}
\end{definition}

We say $E$ is an \emph{admissible} $A$-module of dimension $d$ if it satisfies only (i) and (ii) above.

\section{Characters of Anderson modules}

For each $n \geq 0$ and $B$ any $R$-algebra, let $B^{\phi^n}$ be the $R$-algebra
with structure map $R \stk{\phi^n}{\map} R \map B$. We define the {\it ghost
rings} to be the $R$-algebra given by the product ring $\prod_\phi^n B =
B \times B^\phi \times \cdots B^{\phi^n}$ and $\prod_\phi^\infty B = B \times
B^\phi \times \cdots $. Now as sets define 
$$
W_n(B) = B^{n+1}
$$
and define the set map $w:W_n(B) \map \prod_\phi^n B$ by $w(x_0,\dots, x_n)
= (w_0,\dots, w_n)$ where 
\begin{align}
w_i = x_0^{q^i} + \pi x_1^{q^{i-1}}+ \cdots + \pi^i x_i
\end{align}
are the Witt polynomials. The map $w$ is called the {\it ghost} map.
We define the ring of truncated $\pi$-typical Witt vectors by the following 
theorem as for example in \cite{hessl05}, page $141$:
\begin{theorem}
For each $n\geq 0$, there exists a unique functorial $R$-algebra structure on 
$W_n(B)$ such that $w$ becomes a natural transformation of functors of 
$R$-algebras.
\end{theorem}

Let $(E,\varphi_{E})$ be an admissible $A$-module of dimension $d$ over $S$ and
 $$\varphi_{E}(z)=\displaystyle{\sum_{i=0}^{s}}A_{i}{\tau}^{i}$$
 where $A_{0}=\pi I_{d}+V$ for a nilpotent matrix $V$ of order $d$.
Consider the following exact sequence of $A$-modules
\begin{equation}\label{exact}
0\rightarrow N^{n}\stk{i}{\rightarrow} J^{n}E\rightarrow E\rightarrow 0,
\end{equation}  Note that the $A$-module structure on $N^{n}$ is such that the above  inclusion map $i$ is $A$-linear. 

\subsection{The $A$-linear action of $J^{1}E$ and $N^{1}$} \cblue 
Let $X$ be a $\pi$-formal scheme over $S$. Following the notation in page $807$ in \cite{BS b}, we define $X^{\phi^{n}}$ by $X^{\phi^{n}}(B):=X(B^{\phi^{n}})$ for any $R$-algebra $B$. Then we define
$$\prod^{n}_{\phi}X=X\times_{S}X^{\phi}\times_{S}\ldots \times_{S}X^{\phi^{n}}.$$ 
For any $R$-algebra $B$, consider  the ghost map $w: W_n(B) \map 
\prod^{n}_{i=0} B^{\phi^i}$. 
Hence applying the functor of points $X(~-~)$ 
to the above,
we obtain the following morphism of $\pi$-formal $S$-schemes (which by abuse of
notation is still denoted by $w$)
\begin{align}
\label{ghostX}
w:J^{n}X\longrightarrow \prod^{n}_{\phi}X. 
\end{align}
Suppose $E$ is an admissible $A$-module of dimension $d$ over $S$. Then
for the underlying $\pi$-formal group scheme, we have $E \simeq \hG^d$.
Hence we have the isomorphism $J^nE \simeq W_n^d$ as $\pi$-formal group 
schemes where $W_n$ is the
affine $(n+1)$-space $\hat{\bb{A}}^{n+1}$ with the group structure of the 
additive $\pi$-typical Witt vectors of length $n+1$.
Therefore for any $\pi$-adically complete $R$-algebra $B$ we have $J^nE(B) 
\simeq W_n(B)^d$ and $\prod^n_{\phi} E(B) \simeq \prod^n_{i=0} B^{\phi^i}$.

Hence combining the above with equation (\ref{ghostX}) and taking $n=1$ we 
obtain


\color{black}

$${
\xymatrix{
\left(W_{1}(B)\right)^d\ar[d]_{\varphi_{J^{1}E}(z)} \ar[r]^{w} &B^d\times \left(B^{\phi}\right)^d\ar[d]^{\varphi_{E}(z)\times \varphi_{E^{\phi}}(z)}  \\
\left(W_{1}(B)\right)^d \ar[r]^{w} &B^d\times \left(B^{\phi}\right)^d }}
$$
where $w$ is the ghost map given by $w(\mathbf{x}_{0},\mathbf{x}_{1})=(\mathbf{x}_{0},\mathbf{x}_{0}^{\hq}+\pi \mathbf{x}_{1})$ for $\mathbf{x}_{0},\mathbf{x}_{1}\in B^d$.
Let $\varphi_{J^{1}E}(z)(\mathbf{x}_{0},\mathbf{x}_{1})=(\mathbf{z}_{0},\mathbf{z}_{1})$. Also on the ghost side we have  $$\left(\varphi_{E}(z)\times \varphi_{E^{\phi}}(z)\right)(\mathbf{b}_{0},\mathbf{b}_{1})=\left(\displaystyle{\sum_{i=0}^{s}} A_{i}\mathbf{b}_{0}^{(q^{i})},\displaystyle{\sum_{i=0}^{s}} A^{\phi}_{i}\mathbf{b}_{1}^{(q^{i})}\right)$$
where $\mathbf{b}_{0},\mathbf{b}_{1}\in B^d $, $M^{\phi}=\left(\phi(M_{ij})\right)$, and $M^{(q)}=\left(M^{q}_{ij}\right)$ for any matrix $M=(M_{ij})\in\mathrm{Mat}_{m\times n}(R)$. Then by the commutativity of the diagram we have  $$w(\mathbf{z}_{0},\mathbf{z}_{1})= \left(\varphi_{E}(z)\times \varphi_{E^{\phi}}(z)\right)(\mathbf{x}_{0},\mathbf{x}_{0}^{\hq}+\pi \mathbf{x}_{1}).$$
Therefore $\mathbf{z}_{0}=\displaystyle{\sum_{i=0}^{s}} A_{i}\mathbf{x}_{0}^{(q^{i})}=\varphi_{E}(z)(\mathbf{x}_{0})$ and 

\begin{eqnarray}
\nonumber \mathbf{z}_{0}^{\hq}+\pi \mathbf{z}_{1}&=& \displaystyle{\sum_{i=0}^{s}} A_{i}^{\phi}(\mathbf{x}_{0}^{\hq}+\pi \mathbf{x}_{1})^{(q^{i})}\\
\nonumber \pi \mathbf{z}_{1}&=& \displaystyle{\sum_{i=0}^{s}} (A_{i}^{\phi}-A_{i}^{(\hq)})\mathbf{x}_{0}^{(q^{i+f})}+\displaystyle{\sum_{i=0}^{s}} \pi^{q^{i}}A_{i}^{\phi}\mathbf{x}_{1}^{(q^{i})}\\
\label{Act_N^1}
 \mathbf{z}_{1}&=&\displaystyle{\sum_{i=0}^{s}} A_{i}'\mathbf{x}_{0}^{(q^{i+f})}+\displaystyle{\sum_{i=0}^{s}} \pi^{q^{i}-1}A_{i}^{\phi}\mathbf{x}_{1}^{(q^{i})}
\end{eqnarray}
\cblue where $A_{i}'=\left(a_{lk}'\right)=\left(\delta(a_{lk})\right)$ for any matrix $A_{i}=(a_{lk})\in\mathrm{Mat}_{d}(R)$.  \color{black}
Since $f\geq 1$, by the above calculation we obtain 
\begin{align}\label{action-J^1}
\varphi(J^{1}E)(z)=(\pi I_{2d}+\mathbf{V})\tau^{0}+\mathrm{higher~ order~ terms}
\end{align}
where $\mathbf{V}:= 
\begin{pmatrix} 
V & 0 \\ 
0 & V^{\phi} 
\end{pmatrix} $ is a nilpotent matrix of order $2d$.

\begin{corollary}\label{admsbl-J^n} Let $(E,\varphi_{E})$ be an admissible $A$-module scheme of dimension $d$ then $(J^{n}E,\varphi_{J^{n}E})$ is also an admissible $A$-module scheme of dimension $nd$.
\end{corollary}
 \begin{proof}
 Note that if $E\simeq \hG^{d}$ as $\bF_{q}$ modules then $J^{n}E\simeq \hG^{(n+1)d}$ as $\bF_{q}$ modules. Then for $n=1$ equation \eqref{action-J^1} implies that $J^{1}E$ is admissible. A similar computation will show that $J^{n}E$ is admissible for every $n\geq 0$.
 \end{proof}
\begin{lemma}\label{action_N^{1}} The $A$-module structure of $N^{1}$ induced from $(E,\varphi_{E})$ is given by $$\varphi_{N^{1}}(z)=\displaystyle{\sum_{i=0}^{sf}}\pi^{q^{i}-1}A_{i}^{\phi}{\tau}^{i},$$ where $A_{i}^{\phi}=\left(\phi(a_{lk})\right)$ for any matrix $A_{i}=(a_{lk})\in\mathrm{Mat}_{d}(R)$.
\end{lemma}
\proof  Substituting $\mathbf{x}_{0}=0$ in equation \eqref{Act_N^1} we get the result.        \qed

\begin{corollary} Let $(E,\varphi_{E})$ be an admissible $A$-module scheme of dimension $d$ then $(N^{n}E,\varphi_{N^{n}E})$ is also an admissible $A$-module scheme of dimension $nd$.
\end{corollary}
\begin{proof}
It is clear from  Lemma \ref{action_N^{1}} that the assertion is true for $n=1$. For $n\geq 2$, combining the fact $N^{n}\simeq J^{n-1}N^{1}$ as $A$-module scheme (cf. Theorem 1.2 in \cite{PS-2}) and Corollary \ref{admsbl-J^n}, we obtain our result.
\end{proof}
\bigskip

 Let $E$ be an admissible $A$-module of dimension $d$ over $S$ given by the $A$-action as 
 $$\varphi_{E}(z)=\displaystyle{\sum_{i=0}^{s}}A_{i}{\tau}^{i}$$
 where $A_{0}=\pi I+V(z)$. Note that given an  $\varphi_{E}(ab)=\varphi_{E}(a)\circ \varphi_{E}(b)$. Therefore we obtain
\begin{align}\label{V-id}
\nonumber \theta(ab)I+V(ab)&=(\theta(a)I+V_{a})(\theta(b)I+V(b))\\
V(ab)&=\theta(a)V(b)+\theta(b)V(a)+V(a)V(b)
\end{align}
This in particular shows that $U \circ V(z)=0$ implies $U \circ V(g)=0$ for all $g\in \FF_{q}[z]$, where $U$ is an $R$-linear map. Define 
\begin{align*}
	V^{\perp}:=\{u\in \mathrm{Mat}_{1\times d}(R)~| uV(z)=0\}
\end{align*}
Since $R$ is a discrete valuation ring, $V^{\perp}$ is a free submodule of $R^{d}$ of rank $d-\mathrm{rank}(V(z))$, which we denote as $h$ for the rest of the article.	
Note that $E$ is admissible $A$-module implies $\Lie (E)$ is also admissible as the $A$-action given by $\varphi_{\Lie (E)}(z)=A_{0}{\tau}^{0}.$ \cblue 
Let $v$ denote the $\pi$-adic valuation on $R$. Given any $M=(m_{ij})\in \mathrm{Mat}_{d\times d^{'}}(R)$  the $\pi$-adic valuation of $M$ is defined as $$v(M):=\mathrm{min}\{v(m_{ij}), \ \mathrm{for \ all} \ 1\leq i\leq d,1\leq j\leq d' \}.$$\color{black}

\begin{theorem}\label{inj-lie} Let $E$ be an admissible $A$-module of dimension $d$ over $R$. Then the derivative map $D:\bX(E)\mapsto \bX(\Lie(E))$ is injective.
\end{theorem}
\begin{proof}
Let $\Theta\in \bX(E)$  be an $A$-linear map, and the $A$-linear action of $z$ is given by  $$\varphi_{E}(z)=\displaystyle{\sum_{i=0}^{s}}A_{i}{\tau}^{i}$$
where $z$ is the fixed element of $A$ as introduced before in the beginning of  section $2$.
Since $\FF_{q}[z]\subset A$, without loss of generality, we may assume $A=\FF_{q}[z]$. Let $$\Theta=\displaystyle{\sum_{i=0}^{\infty}}B_{i}{\tau}^{i}$$ where $B_{i}\in \mathrm{Mat}_{1\times d}(R)$ such that $v(B_{i})\rightarrow \infty$ as $i\rightarrow \infty$. As $\Theta$ is $A$-linear, we have the following relation
\begin{align*}
\nonumber \Theta \circ \varphi_{E}(z)&=\varphi_{\hG}(z)\circ \Theta\\
\nonumber\left(\displaystyle{\sum_{i=0}^{\infty}}B_{i}{\tau}^{i}\right)\circ \left(\displaystyle{\sum_{j=0}^{s}}A_{j}{\tau}^{j}\right)&=\pi \left(\displaystyle{\sum_{k=0}^{\infty}}B_{k}{\tau}^{k}\right)\\
\nonumber\displaystyle{\sum_{k}}\left(\displaystyle{\sum_{i+j=k}}B_{j} A_{i}^{(q^{j})}\right){\tau}^{k}&=\displaystyle{\sum_{k}}\pi B_{k}{\tau}^{k}\\
\end{align*}
where $M^{(q)}=\left(M^{q}_{ij}\right)$ for any matrix $M=(M_{ij})\in\mathrm{Mat}_{m\times n}(R)$.
Hence by comparing coefficients of $\tau^k$ we for all $k$,
\begin{align}\label{co-ca}
\displaystyle{\sum_{i=0}^{k}} B_{k-i} A_{i}^{(q^{k-i})}&= \pi B_{k}.
\end{align}

Note that the derivative is given by $D\Theta=B_{0}\tau^{0}$. We will show that $B_{i}$'s are uniquely determined by $B_{0}$ and hence proving $\Theta$ is completely determined by $B_{0}$. We know by the definition of $E$ that $A_{0}=\pi I+V$, where $V$ is a nilpotent matrix of order $d$. We will prove by induction and hence assume that all the $B_{i}$'s, for $0\leq i\leq k-1$ are determined by $B_{0}$. Then from equation \eqref{co-ca}, we have
\begin{align}\label{co-ca-2}
\nonumber \pi B_{k}-B_{k}A_{0}^{(q^{k})}&= \displaystyle{\sum_{i=1}^{k}} B_{k-i} A_{i}^{(q^{k-i})}\\
\nonumber B_{k}(\pi I-A_{0}^{(q^{k})})&= \displaystyle{\sum_{i=1}^{k}} B_{k-i} A_{i}^{(q^{k-i})}\\
B_{k}(cI-V_{0})&=\displaystyle{\sum_{i=1}^{k}} B_{k-i} A_{i}^{(q^{k-i})}
\end{align}
where $c=\pi-\pi^{q^{k}}$ and $V_{0}=V^{(q^{k})}$, which is also another nilpotent matrix. Note that $$c^{d}I=c^{d}I-V_{0}^{d}=(cI-V_{0})V_{1}$$ where $V_{1}=\displaystyle{\sum_{i=0}^{d-1}}c^{d-1-i}V_{0}^{i}$. Multiplying by $V_{1}$ from the right on both sides in equation \eqref{co-ca-2}, we obtain
\begin{align}\label{co-ca-3}
c^{d} B_{k}&=\left(\displaystyle{\sum_{i=1}^{k}} B_{k-i} A_{i}^{(q^{k-i})}\right)V_{1}
\end{align}
Since $c$ is a nonzero element and $R$ is integral domain, $B_{k}$ is uniquely determined in terms of $B_{i}$, for $0\leq i\leq k-1$, hence proving the inductive step and we are done. 
\end{proof}

\begin{lemma}\label{lemma-V-op}Let $B\subset A$ be an $\FF_{q}$-sub-algebra. Assume that $U$ be an $R$-linear map satisfying $U\circ V(b)=0$, for all $b\in B$. Then given a polynomial $P(t)\in B[t]$ we have a polynomial $Q(t)\in R[t]$ with $Q(0)=0$, such that for any $a\in A$ 
 \begin{align*}
U\circ V(P(a))=U\circ (V(a)(P'(\theta (a))I+ Q(V(a))).
\end{align*}
\end{lemma}
\begin{proof}
Note that the operator $V$ is additive. Therefore it is enough to prove the statement for $P(t)=bt^{n}$. Recall from \eqref{V-id} that 
$$V(ba^{n})=\theta(b)V(a^{n})+\theta(a^{n})V(b)+V(b)V(a^{n}).$$
Applying $U$ on both sides of the above equation we get $U\circ V(ba^{n})=U\circ \theta(b)V(a^{n})$. Therefore it is reduced to prove the statement for $P(t)=t^{n}$, which we will do by induction on $n$. Note that the assumption is true for $n=1$ and we assume that $V(a^{n-1})=(n-1)\theta(a)^{n-2}V(a)+Q_{n-1}(V(a))$. We have 
\begin{align*}
V(a^{n})&=\theta(a)V(a^{n-1})+\theta(a^{n-1})V(a)+V(a)V(a^{n-1})\\
        &=\theta(a)((n-1)\theta(a)^{n-2}V(a)+Q_{n-1}(V(a)))+\theta(a^{n-1})V(a)\\
        &+V(a)((n-1)\theta(a)^{n-2}+Q_{n-1}(V(a)))\\
        &=V(a)(n\theta(a)^{n-1})+ Q_{n}(V(a)). 
\end{align*}
This completes the proof.
\end{proof}

\begin{lemma}\label{ext-A-lin} Let $E$ be an admissible $A$-module of dimension $d$ over $R$. If $\Theta\in M(E)$ is $\FF_{q}[z]$-linear then $\Theta$ is $A$-linear.
\end{lemma}
\begin{proof}
Let $\Theta=\displaystyle{\sum_{i=0}^{\infty}}B_{i}{\tau}^{i}$, where $B_{i}\in \mathrm{Mat}_{1\times d}(R)$ such that $v(B_{i})\rightarrow \infty$ as $i\rightarrow \infty$. Then given any $a\in A$ consider the two $\FF_{q}[z]$-linear maps $\Theta\circ \varphi_{E}(a)$ and $\varphi_{\hG}(a)\circ \Theta$. Applying the derivative map on $\Theta\circ \varphi_{E}(a)$ and $\varphi_{\hG}(a)\circ \Theta$, and equating them we get
\begin{align*}
D\Theta\circ (\theta(a)I+V(a))&=\theta(a)D\Theta\\
\theta(a)D\Theta +D\Theta\circ V(a)&=\theta(a)D\Theta\\
D\Theta\circ V(a)&=0.
\end{align*} 
As $\Theta$ is $\FF_{q}[z]$-linear we have $D\Theta \circ V(g)=0$ for all $g\in \FF_{q}[z]$. As $A$ is separable integral extension over $\FF_{q}[z]$, for an $a\in A$, there is a separable polynomial $P(t)\in \FF_{q}[z][t]$ such that $P(a)=0$.  
Then Lemma \ref{lemma-V-op} implies that
\begin{align*}
D\Theta\circ V(P(a))=D\Theta \circ (V(a)(P'(\theta (a))I+ Q(V(a)))=0.
\end{align*}
where $Q(t)$ is a polynomial in $R[t]$ with $Q(0)=0$. As $\theta$ is injective and $P(t)$ is separable, we have $P'(\theta(a))$ is nonzero. Also $V(a)$ nilpotent implies that $Q(V(a))$ is nilpotent. Hence $P'(\theta (a))I+ Q(V(a))$ is an invertible matrix over $K$ that implies $D\Theta \circ V(a)$=0. This shows that the derivative of the two maps $\Theta\circ \varphi_{E}(a)$ and $\varphi_{\hG}(a)\circ \Theta$ agree in $M(\Lie(E))$ for all $a\in A$. Then applying Theorem \ref{inj-lie} for $A=\FF_{q}[z]$, we have $\Theta\circ \varphi_{E}(a)=\varphi_{\hG}(a)\circ \Theta.$ Hence $\Theta$ is $A$-linear. 
\end{proof}

\cblue
\begin{lemma}\label{esti} Let $r_{0}$ be the smallest number such that $q^{r_{0}}>d$. Let $L_{0}\geq r_{0}d$ be an integer. Suppose for each positive integer $k$,  there is an integer  $L_{k}$ satisfying $$L_{k}\geq \min~\{L_{k-i}+(q^{i}-1)q^{k-i}-d, ~\mathrm{for}~i=1,\ldots k\}.$$ Then $L_{k}\geq 0$ and $L_{k}\rightarrow \infty$ as $k\rightarrow \infty$.   
\end{lemma}\color{black}
\begin{proof}
We will first prove that for each $k$, $L_{k}\geq 0$.

\noindent \textbf{\underline{Case-I}:} Let $k\leq r_{0}$. We claim that $L_{k}\geq (r_{0}-k)d$, and we will show it by induction. By our assumption $L_{0}\geq (r_{0}-0)d=r_{0}d$. Assume that the statement holds for numbers $j<k$. Then for $i=1,\ldots k$, we have 

\begin{align*}
L_{k-i}+(q^{i}-1)q^{k-i}-d &\geq L_{k-i}-d\\
&\geq (r_{0}-(k-i))d-d\\
&\geq (r_{0}-k+i-1)d\\
&\geq (r_{0}-k)d\\
&\geq 0.
\end{align*}
  
\noindent\textbf{\underline{Case-II}:} Let $k> r_{0}$. Then by our choice of $r_{0}$, we have $q^{r_{0}}>d$. Hence for $i=1,\ldots k$, we have 
\begin{align*}
L_{k-i}+(q^{i}-1)q^{k-i}-d&\geq (q^{k}-q^{k-i})-d\\
&\geq (q^{k}-q^{k-1})-d\\
&= q^{k-1}(q-1)-d\\
&\geq q^{r_{0}}-d\\
&\geq  0.
\end{align*}
Note that for $k>r_{0}$, the above estimate gives us
\begin{align*}
L_{k-i}+(q^{i}-1)q^{k-i}-d\geq q^{k-1}(q-1)-d,~ \mathrm{for~ all}~ i=1,\ldots k.
\end{align*}
Hence 
\begin{align*}
L_{k}&=\min \{L_{k-i}+(q^{i}-1)q^{k-i}-d ~|~ i=1,\ldots k\}\\
&\geq q^{k-1}(q-1)-d.
\end{align*}
Therefore $L_{k}\rightarrow \infty$ as $k\rightarrow \infty$ and that completes the proof.
\end{proof}

\begin{theorem}\label{char-N^1} Let $E$ be an admissible $A$-module of dimension $d$ over $R$. Let $r_{0}$ be the smallest number such that $q^{r_{0}}>d$. If $B_{0}\in \pi^{r_{0}d}\bX(\Lie(E))$, then there exists a unique $A$-linear map $\nu_{1}:N^{1}\mapsto \hG$ such that $D\nu_{1}=B_{0}$.
\end{theorem}
\begin{proof}
By Lemma \ref{ext-A-lin}, it is enough to show the map $\nu_1$ is $\FF_{q}[z]$-linear. 
Using Lemma \ref{action_N^{1}} and a similar calculation as in \eqref{co-ca-3}, we have 
$$c^{d} B_{k}=\left(\displaystyle{\sum_{i=1}^{k}} B_{k-i} (\pi^{q^{i}-1}{A^{\phi}}_{i})^{(q^{k-i})}\right)V_{1}$$
Therefore computing the order with respect to $\pi$ on both sides of the above equation, we get 
\begin{align*}
\nonumber v(B_{k})+d&\geq \min\left\{v\left(\pi^{(q^{i}-1)(q^{k-i})}B_{k-i}{A^{\phi}}_{i}^{(q^{k-i})}\right)~|~ 1\leq i\leq k\right\} \\
v(B_{k})&\geq \min\{v(B_{k-i})+(q^{i}-1)(q^{k-i})-d ~|~ 1\leq i\leq k\}
\end{align*}
Since $v(B_{0})\geq r_{0}d$, by Lemma \ref{esti}, we have $v(B_{k})\geq 0$ and $v(B_{k})\rightarrow \infty$ as $k\rightarrow \infty$. This allow us to define $B_{k}$ as above. Therefore the $A$-linear map $\nu_{1}=\displaystyle{\sum_{i=0}^{\infty}}B_{j}\tau^{j}$ exists and Theorem \ref{inj-lie} shows it is unique.
\end{proof}

\begin{proposition}\label{rank-lie} Let $E$ be an abelian Anderson $A$-module of dimension $d$ and rank $r$ over $R$. Then $\bX(\Lie E)\simeq V^{\perp}$. In particular, $\mathrm{rank}_{R}\bX(\Lie E)=h$.
\end{proposition}
\begin{proof}
Let $\displaystyle{\sum_{i\geq 0}} B_{i}\tau^{i}\in \bX(\Lie E)$, where $B_{i}\in \mathrm{Mat}_{1\times d}(R)$. Then from equation \eqref{co-ca} we have

$$\displaystyle{\sum_{i=0}^{k}} B_{k-i} A_{i}^{(q^{k-i})}= \pi B_{k}.$$
For $k=0$, we get $B_{0}V=0$. Hence $B_{0}\in V^{\perp}$. 
When $k>0$, we have 
\begin{align*}
B_{k}A_{0}^{(q^k)}=\pi B_{k} \\
B_{k}((\pi^k -\pi)I + V^{(q^k)})=0
\end{align*}
Note that the matrix $V$ is nilpotent implies $V^{(q^k)}$ is again a nilpotent matrix. Hence for $k>0$, the matrix $(\pi^k -\pi)I + V^{(q^k)}$ is invertible in $K$. Therefore $B_{k}=0$ for $k>0$. Hence  by identifying $B_{0}\tau^{0}\mapsto B_{0}$, we conclude that $\bX(\Lie E)\subset V^{\perp}$.

Conversely, any element $B_{0}\in V^{\perp}$ will satisfy \eqref{co-ca} for $k=0$. Therefore $B_{0}\tau^{0}\in \bX(\Lie E)$. Thus $B_{0}\tau^{0}\mapsto B_{0}$ gives the required identification.
\end{proof}

\begin{corollary}\label{rank-N^1} Let $E$ be an abelian Anderson $A$-module of dimension $d$ and rank $r$ over $R$. Then the rank of the $R$-module $\bX(N^{1})$ is $h$.
\end{corollary}

\begin{proof} Note that by Lemma \ref{action_N^{1}}, we have $\Lie N^{1}\simeq \Lie E$ as $A$-module schemes. Also by \eqref{action_N^{1}}, we observe that the admissible $A$-module scheme $N^{1}$ satisfies the hypothesis of Theorem \ref{char-N^1}.
Therefore from Theorem \ref{inj-lie} and Theorem \ref{char-N^1} it follows that $$\pi^{r_{0}d}\bX(\Lie N^{1})\subset \bX(N^{1})\subset \bX(\Lie(N^{1}).$$
Therefore, invoking upon Proposition \ref{rank-lie}, we have $$\mathrm{rank}_{R}\bX(N^{1})=\mathrm{rank}_{R}\bX(\Lie N^{1}))=\mathrm{rank}_{R}\bX(\Lie E)=h.$$
This completes the proof.
\end{proof}

\section{Lateral Frobenius}
We briefly recall here the lateral Frobenius maps constructed in \cite{BS b}.
Let $\mathbb{W}_{n}$ denote the $\pi$-formal group scheme isomorphic to $\hat{\mathbb{A}}^{n+1}$ with the group structures of the additive Witt vectors of length $n+1$.
Then consider the following diagram (not commutative)
\begin{equation}\label{Lat_witt}
{
\xymatrix{
\mathbb{W}_{n}\ar[d]_{F} \ar[r]^{V} &\mathbb{W}_{n+1}\ar[d]^{F}  \\
\mathbb{W}_{n-1} \ar[r]^{V} &\mathbb{W}_{n}\ar[d]^{F}\\
&\mathbb{W}_{n-1} }}
\end{equation}
However we have the following relation  
\begin{equation}\label{Fr_var}
FFV=FVF.
\end{equation}

The above equation is true because the operator $FV$ is multiplication by $\pi$, and the Frobenius $F$ commutes with $\pi$. We can rephrase the above diagram in terms of jet spaces using the natural identifications $J^{n}E\simeq \mathbb{W}_{n}^{d}$ and $N^{n}\simeq \mathbb{W}_{n-1}^{d}$ as group schemes. For jet spaces, we use the notation $i := V$ and $\phi := F$ for the right column of the diagram above. Then we define the \emph{lateral Frobenius}
\begin{equation*}
\mathfrak{f}:N^{n+1}\rightarrow N^{n}
\end{equation*}
to be simply the map $F:\mathbb{W}_{n}\rightarrow \mathbb{W}_{n-1}$ under the above identification. Thus diagram \eqref{Lat_witt} can be reexpressed as

\begin{equation}\label{Lat_jet}
{
\xymatrix{
N^{n+1}\ar[d]_{\mathfrak{f}} \ar[r]^{i} &J^{n+1}E\ar[d]^{\phi}  \\
N^{n} \ar[r]^{i} &J^{n}E\ar[d]^{\phi}\\
&J^{n-1}E }}
\end{equation} 
Here we emphasize that when we use the notation $N^{n}$, the $A$-module structure will always be understood to be the one that makes $i$ an $A$-linear morphism. It is not be confused with the $A$-module structure coming by transport of structure from the isomorphism $N^{n}\simeq \mathbb{W}_{n-1}^{d} \simeq
 J^{n-1}E$ of group schemes.
We also remark that while $i$ is a morphism of $S$-schemes, the vertical arrows $\phi$  and $\mathfrak{f}$ in diagram \eqref{Lat_jet} lie over the Frobenius endomorphism  of $S$, rather than the identity morphism.\medskip

Although the lateral Frobenius $\mathfrak{f}:N^{n}\rightarrow N^{n-1}$ is a
morphism of group schemes, following a similar argument as Proposition 7.2 in \cite{BS b}, they are $A$-linear.
The lateral Frobenius $\mathfrak{f}$ along with the projection map $u$  make the system $\{N^{n}\}_{n\geq 0}$ into a prolongation sequence. 

\begin{lemma}\label{lf identity}
For $n\geq 2$, we have $\phi^{\circ n}\circ i=\phi\circ i\circ \mathfrak{f}^{n-1}$.
\end{lemma}
\proof Enough to prove for $n=2$, which follows from equation \eqref{Fr_var}.\qed

%
%

\begin{theorem}\label{basis_canonical_char} ${\bX}(N^{n})$ is a free 
$R$-module of rank $nh$.

\end{theorem}
\proof Note that by Corollary \ref{rank-N^1}, we have $\bX(N^{1})\simeq R^{h}$.   
 Then a similar argument as in Proposition $7.4$ in \cite{BS b} will show that ${\bX}(N^{n})$ is a free $R$-module of rank $nh$.
\qed
\section{The $R\{\phi^*\}$-module ${\bX}_{\infty}(E)$}
%
%
\subsection{Splitting of $J^{n}E$}\label{splitting_J^n}
Note that the short exact sequence \eqref{exact} always splits as $\mathbb{F}_{q}$-module schemes by the \cblue \textup{Teichm\"{u}ller section}$$\teich:E\longrightarrow J^{n}E .$$ Therefore as an $\mathbb{F}_{q}$-module scheme $$J^{n}E \simeq E\times N^{n}.$$ Given any element $\Theta \in {\bX}_{n}(E)$, we have the following commutative diagram 
$${
\xymatrix{
E \ar[r]^\teich \ar[rd]_{\teich^{*}\Theta} &
J^{n}E \ar[d]^\Theta &
N^{n} \ar[l]_i \ar[ld]^{i^{*}\Theta}\\
&\hat{\mathbb{G}}_a}
}.$$
Therefore we can write $\Theta(\mathbf{x}_{0},\ldots \mathbf{x}_{n})=\teich^{*}\Theta(\mathbf{x}_{0})+i^{*}\Theta(\mathbf{x}_{1},\ldots ,\mathbf{x}_{n})$. We would like to remark here that the above splitting of $J^{n}E$ is not as
$A$-modules schemes in general. The map $\teich^{*}\Theta$ is  $\mathbb{F}_q$-linear only. The following theorems are straightforward generalization of
 Section $8$ of \cite{BS b}. \color{black}
\begin{theorem}\label{one parameter} There exists a unique $A$-linear map $h$ commuting the following diagram

$$\xymatrix{
N^{n+1} \ar[rr]^-{\phi \circ i-i \circ \mathfrak{f}}  \ar[d]_u & & J^{n}E \\
N^1 \ar@{..>}[rru]_h & &
}$$
\end{theorem}
\proof 
See Proposition 8.4 in \cite{BS b}. 
\qed
\begin{proposition}\label{Fr identity} Let $\Theta$ be a character in  ${\bX}_{n}(E)$.
\begin{itemize}
\item[(i)] We have $$i^{*}\phi^{*}\Theta=\mathfrak{f}^{*}i^{*}\Theta + \gamma. \Psi_{1},$$ where $\Psi_{1} \in {\bX}(N^{1})$ is the canonical character and $\gamma\in R^{d}$ is divisible by $\pi$.\medskip

\item[(ii)] For $n\geq 2$, we have $$i^{*}(\phi^{\circ n})^{*}\Theta = (\mathfrak{f}^{n-1})^{*}i^{*}\phi^{*}\Theta.$$
\end{itemize} 

\end{proposition}
\proof Follows similarly as in Proposition $8.5$ in \cite{BS b}.   \qed

\begin{theorem}\label{frob_bij}
 For $n\geq 1$, we have the following commutative diagram
 $${
\xymatrix{
{\bX}_{n}(E)/{\bX}_{n-1}(E)\ar@{^{(}->}[d]_{i^{*}} \ar@{^{(}->}[r]^{\phi^{*}} &{\bX}_{n+1}(E)/{\bX}_{n}(E)\ar@{^{(}->}[d]^{i^{*}}  \\
{\bX}(N^{n})/{\bX}(N^{n-1}) \ar[r]^{\mathfrak{f}^{*}}_{\sim} &{\bX}(N^{n+1})/{\bX}(N^{n}) }}
$$
with $i^{*}$ and $\phi^{*}$ injective and $\mathfrak{f}^{*}$ bijective . 
\end{theorem}
\proof Follows similarly as in Proposition 8.6 in \cite{BS b}. \qed

\section{Extensions of abelian Anderson modules and de Rham Cohomology}
Here we recall some of the general theory of extensions of $A$-module schemes as in Section $5$ of \cite{Gekeler_a}. Let $(E,\varphi_{E})$ and $(E',\varphi_{E'})$ be two admissible $A$-modules of dimension $d$ and $l$, respectively. Let $M:=\mathrm{Hom}_{R,\mathbb{F}_q}(E,E')$ and $N:=M\tau=\{m\in M~|~\mathrm{Lie}(m)=0\}$ be an $A_{R}$ submodule of $M$. Note that when $(E',\varphi_{E'})=(\hat{\mathbb{G}}_{a},\varphi_{\hat{\mathbb{G}}_a})$, $M$ is the associated $A$-motive $M(E)$ of $E$.

\begin{definition} A $(\varphi_{E},\varphi_{E'})$-biderivation is an $\mathbb{F}_{q}$-linear map $\eta:A\longrightarrow N$ satisfying the product rule $$\eta(ab)=\varphi_{E'}(a)\eta(b)+\eta(a)\varphi_{E}(b) ~\mathrm{for~ all} ~a,b \in A.$$
\end{definition}
We denote the set of all $(\varphi_{E},\varphi_{E'})$-biderivations as $D(\varphi_{E},\varphi_{E'})$ and we will refer to them as biderivations only. 
For each $m\in M$, we can associate an $\FF_{q}$-linear map $\eta^{m}:A\longrightarrow M$ as $$\eta^{m}(a):=\varphi_{E'}(a)m-m\varphi_{E}(a).$$ If the image of $\eta^{m}$ lies in $N$, we call it inner (resp. strictly inner if $m\in N$) biderivations and denoted as $D_{i}(\varphi_{E},\varphi_{E'})$ (resp. $D_{si}(\varphi_{E},\varphi_{E'})$). Whenever  $(E',\varphi_{E'})=(\hat{\mathbb{G}}_{a},\varphi_{\hat{\mathbb{G}}_a})$, we will denote the
above modules as $D_i(\varphi_E)$ and $D_{si}(\varphi_E)$ respectively.

\noindent Recall from Section $5$ of \cite{Gekeler_a} that $\mathrm{Ext}_{A}(E,E')$ denotes the group of extension classes of $A$-module schemes of $E$ by $V$ over $R$. Moreover $\mathrm{Ext}^{\#}_{A}(E,E')$ denotes the group of extension classes of $A$-module schemes of $E$ by $E'$ over $R$ together with a splitting of the 
induced extension of Lie algebras.
Although in \cite{Gekeler_a} only the extensions by $E'={\hat{\mathbb{G}}_a}$ has been considered but the similar theory can be developed in general. Given any $\eta \in D(\varphi_{E},\varphi_{E'})$ consider the matrix

\begin{equation*}
\varphi_{\eta}(a) = 
\begin{pmatrix}
\varphi_{E'}(a) & \eta(a) \\
0 & \varphi_{E}(a) 
\end{pmatrix}.
\end{equation*}
We can associate an extension of $(E,\varphi_{E})$ by $(E',\varphi_{E'})$ of $A$-module schemes over $S$ as 
\begin{equation*}
0\longrightarrow E' \longrightarrow E'\oplus E \longrightarrow E \longrightarrow 0,
\end{equation*}
where the $A$-module action on $E'\oplus E$ is given by the above matrix $\varphi_{\eta}(a)$. Since image of $\eta $ lies in $N$, the induced exact sequence on Lie algebras will split. Therefore there is a canonical map $$\beta^{\#}:D(\varphi_{E},\varphi_{E'})\longrightarrow \mathrm{Ext}^{\#}_{A}(E,E')$$  $$\eta \longrightarrow ([\eta],s) .$$
This will induce another map $$\beta:D(\varphi_{E},\varphi_{E'})\longrightarrow  \mathrm{Ext}^{\#}_{A}(E,E') \longrightarrow  \mathrm{Ext}_{A}(E,E')$$ $$ \eta \longrightarrow ([\eta],s)\longrightarrow [\eta].$$

\begin{lemma}\label{ext_bider} We have 

\begin{enumerate}
\item[(i)]The map $\beta^{\#}$ is surjective.
\item[(ii)] $\mathrm{Ker}~(\beta)=D_{i}(\varphi_{E},\varphi_{E'})$.
\item[(iii)]$\mathrm{Ker}~(\beta^{\#})=D_{si}(\varphi_{E},\varphi_{E'})$.
\end{enumerate}
\end{lemma}
\proof   Let $G$ be an $A$-module scheme and                                                                        
$0\longrightarrow E' \longrightarrow G \longrightarrow E \longrightarrow 0$ be an extension of $E$ by $E'$ with a splitting $s:\mathrm{Lie}(E)\longrightarrow \mathrm{Lie}(G)$. Since $E$ and $E'$ are isomorphic to their Lie algebras as group scheme and $\mathrm{Lie}(G)\simeq \mathrm{Lie}(E')\oplus \mathrm{Lie}(E)$, we have $G\simeq \mathrm{Lie}(G)$ as group schemes, hence $G\simeq E'\oplus E$ as group schemes (not as $A$-module schemes apriori). Thus we have the following split exact sequence of group schemes  $$0\longrightarrow E' \longrightarrow E'\oplus E \longrightarrow E \longrightarrow 0.$$ 
Hence the $A$-module structure on $G$, under the isomorphism with $E'\oplus E$ can be written as
\begin{equation*}
\varphi_{G}(a):= 
\begin{pmatrix}
\varphi_{E'}(a) & \eta(a) \\
0 & \varphi_{E}(a) 
\end{pmatrix} \ \mathrm{for ~ all}~ a\in A.
\end{equation*} 
Then from the splitting of Lie algebras and the associativity of $A$-action implies that $\eta$ is indeed a biderivation. This shows that $\beta^{\#}$ is surjective and proves (i). \medskip

Given an $\eta^{m} \in D_{i}(\varphi_{E},\varphi_{E'})$, we can define the $A$-module splitting map $$s_{\eta^{m}}:E \longrightarrow E'\oplus E $$  $$x \mapsto (-m(x),x).$$ 
On the other hand given any $A$-module splitting 
$$s:E \longrightarrow E'\oplus E $$  $$x \mapsto (m(x),x),$$ 
we have $m\in M$ and $s$ commutes with the $A$-action which implies $\eta=\eta^{(-m)}$. Therefore kernel of $\beta $ is $D_{i}(\varphi_{E},\varphi_{E'})$ proving (ii).\medskip
 
\noindent Similarly, we can show that given any trivial extension of $E$ by $E'$  with trivial splitting on Lie algebras corresponds to an element $\eta \in D_{si}(\varphi_{E},\varphi_{E'})$, i.e kernel of $\beta^{\#}$ is $D_{si}(\varphi_{E},\varphi_{E'})$ proving (iii).
\qed
\subsection{Functoriality}
Let $(V',\varphi_{V'})$ be another strict $A$-module scheme with an $A$-module map $f:E' \longrightarrow V'$. Then we have a map 
$$f_{*}:D(\varphi_{E},\varphi_{E'})\longrightarrow D(\varphi_{E},\varphi_{V'})$$ 
$$\eta \mapsto f\circ \eta.$$ 
Under the above map $f_*$, we have that $D_{i}(\varphi_{E},\varphi_{E'})$ goes to $D_{i}(\varphi_{E},\varphi_{V'})$ and $D_{si}(\varphi_{E},\varphi_{E'})$ to $D_{si}(\varphi_{E},\varphi_{V'})$. Therefore by Lemma \ref{ext_bider}  we have the following map 
\begin{equation*}
f_{*}:  \mathrm{Ext}_{A}(E,E') \longrightarrow  \mathrm{Ext}_{A}(E,V'),
\end{equation*}
and 
\begin{equation*}
f_{*}: \mathrm{Ext}^{\#}_{A}(E,E')\longrightarrow  \mathrm{Ext}^{\#}_{A}(E,V').
\end{equation*}
In this way, we get an extension $f_{*}G$ of $E$ by $V'$ which is the push out by $f$ as follows:
\begin{align}\label{pushout}
\xymatrix{
0  \ar[r] &E' \ar[d]^{f} \ar[r] &G \ar[d]^{e_{f}} \ar[r] &E \ar@{=}[d] \ar[r]&0 \\
0 \ar[r] &V' \ar[r] &f_{*}G \ar[r] & E \ar[r] &0}.
\end{align}
Here we recall the definition of de Rham cohomology of an Anderson module introduced in Section 3 of \cite{BrPa}.
\begin{definition}\label{deRhamdef} Let $(E,\varphi_{E})$ be an abelian Anderson $A$-module. The de Rham cohomology of $E$ is defined as the $R$-module 
\begin{equation*}
\bH^{*}_{\mathrm{dR}}({E}):
=\mathrm{Ext}^{\#}_{A}(E,\hat{\mathbb{G}}_{a})=D(\varphi_{E})/D_{si}(\varphi_{E}).
\end{equation*}
\end{definition}

\begin{theorem}\label{rank_deRham} 
Let $E$ be an abelian Anderson $A$-module of rank $r$ over $R$. Then the de Rham cohomology module $\bH^{*}_{\mathrm{dR}}({E})$ is a free $R$-module of rank $r$. 

\end{theorem}
\begin{proof}
By definition the associated $A$-motive $M(E)$ of $E$ is an $A_{R}=A\otimes_{\mathbb{F}_{q}} R$-module, which is locally free of rank $r$ defined by
$$(a\otimes b)m:=b\circ m\circ \varphi_{E}(a),\quad \mathrm{for \ any \ a\otimes b \in A\otimes_{\mathbb{F}_{q}}R}.$$ 
Now consider the ideal $I$ generated by elements of the form $a\otimes 1-1\otimes a$ in $A\otimes_{\mathbb{F}_{q}} A$. Let $J$ be the kernel of the following map \begin{align*}
A_{R}=A\otimes_{\FF_{q}} R&\longrightarrow R\\
a\otimes r &\mapsto \theta(a)r
\end{align*}
Then $J=I\otimes_{A\otimes_{\mathbb{F}_{q}} A} A_{R}$. As $X/\FF_{q}$ is smooth curve, we have $I/I^{2}\simeq \Omega^{1}_{A/\FF_{q}}$ is a locally free $A$-module of rank $1$. This implies $J/J^{2}$ is a free $R$-module of rank $1$. 
Note that since $M(E)$ is locally free $A_{R}$-module, that implies $N(E)$ is $A\otimes_{\FF_{q}} A$-torsion free as $\theta$ is injective. 
Then by the results in Section $4.2$ and $4.3$ in \cite{Gekeler_b}, we have the following identifications
\begin{align*}
D(\varphi_{E})\simeq\Hom_{A_{R}}(J, N(E)) ~\mathrm{and} ~D_{si}(\varphi_{E})\simeq J\Hom_{A_{R}}(J, N(E))
\end{align*}
Also note that the right multiplication by $\tau$ map from $M(E)$ to $N(E)$ is an $R$-linear isomorphism. However it is not $A$-linear. This induces an isomorphism in the quotients between $M(E)/JM(E)$ and $N(E)/JN(E)$ as an $R$-module. Therefore we have $M(E)/JM(E)\simeq N(E)/JN(E)\simeq R^{r}$ as an $R$-module. Since $J$ is invertible ideal of $A_{R}$ and $A_{R}/J\simeq R$, we get
\begin{align*}
\bH^{*}_{\mathrm{dR}}({E})&=D(\varphi_{E})/D_{si}(\varphi_{E})\\
& \simeq \Hom_{A_{R}}(J, N(E))/J\Hom_{A_{R}}(J, N(E))\\
&\simeq \Hom_{A_{R}}\left(J, N(E)\right)\otimes _{A_{R}} A_{R}/J\\
&\simeq \Hom_{A_{R}}\left(J, N(E)/JN(E)\right)\\
&\simeq \Hom_{A_{R}}\left(J, A_{R}/J\right)^{r}\\
&\simeq \Hom_{A_{R}}\left(J,\Hom_{A_{R}/J}\left(A_{R}/J,A_{R}/J\right)\right)^{r}\\
&\simeq \Hom_{A_{R}/J}\left(J\otimes_{A_{R}}{A_{R}/J}, A_{R}/J\right)^{r}\\
&\simeq \Hom_{R}(J/J^{2}, R)^{r}\\
&\simeq R^{r}.
\end{align*}
\end{proof}

\begin{lemma}\label{bider-id} Let $B\subset A$ be an $\FF_{q}$-sub-algebra. Assume that $\eta(b)=0$, for all $b\in B$. Then given any $P(t)\in B[t]$, we have an $\mathfrak{n}\in N(E)$ and $Q(t)\in R[t]$ with $Q(0)=0$ such that for any $a\in A$
$$\eta(P(a))= \eta(a)\circ((P'(\theta (a))I+Q(V(a)))\tau^{0}+\eta(a)\circ\mathfrak{n}.$$
\begin{proof}
Note that $\eta$ is additive. Therefore it is enough to prove the statement for $P(t)=bt^{n}$. By definition we have 
\begin{align*}
\eta(ba^{n})=\theta(b)\eta(a^n)+\eta(b)\circ\varphi_{E}(a^{n})= \theta(b)\eta(a^n),~\mathrm{as}~\eta(b)=0.
\end{align*}
Hence the proof is reduced for $P(t)=t^{n}$, which we will do by induction on $n$. Note that for $n=1$, the assumption is true for $Q(t)$ and $\mathfrak{n}$ to be zero. Assume the statement is true for $n$, i.e,
$$\eta(a^n)= \eta(a)\circ(n\theta (a)^{n-1}I+Q_{n}(V(a))\tau^{0}+\eta(a)\circ\mathfrak{n}_{n}.$$
Then 
\begin{align*}
\eta(a^{n+1})&=\theta(a)\eta(a^{n})+\eta(a)\circ\varphi_{E}(a^{n})\\
             &=\theta(a)\eta(a^{n})+\eta(a)\circ(\theta(a^n)I+V(a^{n}))\tau^{0}+ \eta(a)\circ\mathfrak{n}\\
            &=\theta(a)\eta(a)\circ(n\theta (a)^{n-1}I+Q_{n}(V(a)))\tau^{0}+\theta(a)\eta(a)\circ\mathfrak{n}_{n}\\
             &+\eta(a)\circ(\theta(a^n)+V(a^{n}))\tau^{0}+ \eta(a)\circ\mathfrak{n}\\
             &=\eta(a)\circ(((n+1)\theta(a)^{n}I+\theta(a)Q_{n}(V(a))+V(a^{n})))\tau^{0}\\
             &+\eta(a)\circ (\mathfrak{n}+\mathfrak{n}_{n}\circ \theta(a)I\tau^{0})\\
             &=\eta(a)\circ((n+1)\theta(a)^{n}I+Q_{n+1}(V(a)))\tau^{0}+\eta(a)\circ \mathfrak{n}_{n+1}.
\end{align*}
This proves the claim.
\end{proof}
\end{lemma}
\begin{lemma}\label{res-inj}  Let $E$ be an abelian Anderson $A$-module over $R$. Let $B\subset A$ be a subring such that $A$ is separable integral extension of $B$. Then the restriction map $D_{A}(\varphi_{E})\longrightarrow D_{B}(\varphi_{E})$ is injective. 
\end{lemma}
\begin{proof} Assume that $\eta(b)=0$ for all $b\in B$. Given any $a\in A$, there is a separable monic polynomial $P(t)\in B[t]$ such that $P(a)=0$. Then using Lemma \ref{bider-id}, we have
\begin{align}\label{eqn-eta}
\eta(P(a))= \eta(a)\circ(P'(\theta (a))I+Q(V(a)))\tau^{0}+\eta(a)\circ\mathfrak{n}=0
\end{align}
where $\mathfrak{n}\in M(E)\tau$. Since $P(t)$ is separable and $\theta$ is injective, we have $P'(\theta (a)$ is nonzero. Let $\eta(a)=\displaystyle{\sum_{i=0}^{\infty}}B_{j}\tau^{j}\in M(E)$. Assume that $\eta(a)$ is non-zero and $j_{0}$ be the smallest integer such that $B_{j_{0}}\neq 0$.
Then comparing coefficients of $\tau^{j_{0}}$ on both sides of \eqref{eqn-eta}, we have
$$ B_{j_{0}}(P'(\theta (a))I+Q(V(a)))^{(q^{j_{0}})}=0$$
Since $Q(t)\in R[t]$ with $Q(0)=0$, thus $V(a)$ nilpotent implies $Q(V(a))$ is nilpotent. Therefore $(P'(\theta (a))I+Q(V(a)))^{(q^{j_{0}})}$ is an invertible matrix over $K$, that implies $B_{j_{0}}=0$, which is a contradiction. Therefore $\eta(a)=0$ for all $a\in A$ and that proves the restriction map is injective. 
\end{proof}

\begin{lemma}\label{inject} We have ${D_{i}}_{A}(\varphi_{E})\cap {D_{si}}_{B}(\varphi_{E})={D_{si}}_{A}(\varphi_{E}).$
\end{lemma}

\begin{proof} Note that by the definition and Lemma \ref{res-inj} we have ${D_{si}}_{A}(\varphi_{E})\subset {D_{i}}_{A}(\varphi_{E})\cap {D_{si}}_{B}(\varphi_{E})$. Let $\eta=\eta^{m}=\eta^{n}\in {D_{i}}_{A}(\varphi_{E})\cap {D_{si}}_{B}(\varphi_{E})\subset {D_{i}}_{B}(\varphi_{E})$ for some $m\in M(E)$ and $n\in N(E)$. Then $\eta^{m-n}=\eta^{m}-\eta^{n}=\eta-\eta=0 \in {D_{i}}_{B}(\varphi_{E})$. Note that $M(E)$ is torsion free as $A_{R}$-module and $\eta^{m-n}(b)=(b-\theta(b))(m-n)$. Therefore, $\eta^{m-n}(b)=\varphi_{\hG}(b)(m-n)-(m-n)\varphi_{E}(b)=(b-\theta(b))(m-n)=0$, for all $b\in B$. Note that $b-\theta(b)$ generates the nonzero invertible ideal $I_{B}\subset B_{R}$. This implies that $m=n$. Hence $\eta\in {D_{si}}_{A}(\varphi_{E})$. This completes the proof. 

\end{proof}

\begin{lemma}\label{inject-inn} We have ${D}_{A}(\varphi_{E})\cap {D_{i}}_{B}(\varphi_{E})={D_{i}}_{A}(\varphi_{E})$ and ${D}_{A}(\varphi_{E})\cap {D_{si}}_{B}(\varphi_{E})={D_{si}}_{A}(\varphi_{E})$
\end{lemma}

\begin{proof} We will prove the first assertion only, later will follow by similar arguments. Note that by the definition and Lemma \ref{res-inj} we have ${D_{i}}_{A}(\varphi_{E})\subset D_{A}(\varphi_{E})\cap {D_{i}}_{B}(\varphi_{E})$. Let $\eta\in {D}_{A}(\varphi_{E}) $ be an $A$-biderivation such that $\eta|_{B}=\eta^{m}\in {D_{i}}_{B}(\varphi_{E})$ for some $m\in M(E)$. This implies $\theta(b)m-m\varphi_{E}(b)\in N(E)$ for all $b\in B$. By Lemma \ref{res-inj}, it is enough to prove that $\eta^{m}\in {D_{i}}_{A}(\varphi_{E})$, i.e, we need to prove $\eta^{m}(a)\in N(E)$, for all $a\in A$. Note that $\eta^m=\eta^{u}+\eta^{n}$ for $n\in N(E)$ and $u\in R^{d}\tau^{0}$. Moreover, $\eta^{n}\in {D_{i}}_{A}(\varphi_{E})$ directly following the definition. 

Therefore it is enough to consider a $B$-inner derivation of the form $\eta^{u}$, where $u\in R^{d}\tau^{0}$. Then $\eta^u(z)\in N(E)$ implies that $uV(z)=0$. Then by Lemma \ref{lemma-V-op} and Lemma \ref{ext-A-lin} we have $uV(a)=0$. As $uV(a)=0$, we have $\eta^u(a)\in N(E)$ and consequently it becomes an $A$-inner derivation. Therefore $\eta^m$ extends to an $A$-biderivation that completes the proof.  
\end{proof}
\begin{lemma}\label{iso-inn} The natural map induced by restriction $${D_{i}}_{A}(\varphi_{E})/{D_{si}}_{A}(\varphi_{E})\longrightarrow {D_{i}}_{B}(\varphi_{E})/{D_{si}}_{B}(\varphi_{E})$$ is an isomorphism of $R$-modules. 
\end{lemma}
\begin{proof}
By Lemma \ref{res-inj} and Lemma \ref{inject} implies that the natural restriction map in the quotients is injective. Therefore it is enough to prove that the map is surjective. Since $\eta^{m+n}=\eta^{m}+\eta^{n}$ for $m,n\in M(E)$, it is enough to consider a $B$-inner derivation $\eta=\eta^{u}$, where $u\in R^{d}\tau^{0}$. Then $\eta(z)\in N(E)$ implies that $uV(z)=0$. Then by Lemma \ref{lemma-V-op} and Lemma \ref{ext-A-lin} we have $uV(a)=0$. Define $\tilde{\eta}(a):=\eta^{u}(a)$ for all $A$. As $uV(a)=0$, we have $\tilde{\eta}(a)\in N(E)$ and consequently it becomes an $A$-inner derivation. Therefore $\eta$ extends to an $A$-biderivation that completes the proof. 
\end{proof}

The following is the general case of the result proved in \cite{BrPa}. 

\begin{theorem}\label{rank_Ext} Let $E$ be an abelian Anderson $A$-module of rank $r$ over $R$. Then the $R$-module $\mathrm{Ext}_{A}(E,\hat{\mathbb{G}}_{a})=D(\varphi_{E})/D_{i}(\varphi_{E})$ is free of rank $r-h$.
\end{theorem}
\begin{proof}When $A=\FF_{q}[z]$, the assertion follows from Proposition $3.1.3$ in \cite{BrPa}. For general $A$, we consider the subring $B=\FF_{q}[z]\subset A$ and $A$ is separable integral extension of $B$. Then we have a map between the following short exact sequences of $R$-modules given by restrictions
\begin{align*}
\xymatrix{
0  \ar[r] &{D_{i}}_{A}(\varphi_{E})/{D_{si}}_{A}(\varphi_{E}) \ar@{=}[d] \ar[r] &D_{A}(\varphi_{E})/{D_{si}}_{A}(\varphi_{E}) \ar@{^{(}->}[d] \ar[r] &D_{A}(\varphi_{E})/{D_{i}}_{A}(\varphi_{E}) \ar@{^{(}->}[d] \ar[r]&0 \\
0 \ar[r] &{D_{i}}_{B}(\varphi_{E})/{D_{si}}_{B}(\varphi_{E}) \ar[r] &D_{B}(\varphi_{E})/{D_{si}}_{A}(\varphi_{E}) \ar[r] & D_{B}(\varphi_{E})/{D_{i}}_{B}(\varphi_{E}) \ar[r] &0}
\end{align*}
The vertical inclusions in the above diagram follows from Lemma \ref{inject-inn}. As $R$ is a discrete valuation ring, the above diagram shows that $\mathrm{Ext}_{A}(E,\hat{\mathbb{G}}_{a})=D_{A}(\varphi_{E})/{D_{i}}_{A}(\varphi_{E})$ is a free $R$-module being a submodule of $\mathrm{Ext}_{B}(E,\hat{\mathbb{G}}_{a})$, which is free by Proposition $3.1.3$ in \cite{BrPa}.  
Moreover, by Proposition $3.1.3$ in \cite{BrPa} and by Lemma \ref{iso-inn}, we have $${D_{i}}_{A}(\varphi_{E})/{D_{si}}_{A}(\varphi_{E})\simeq {D_{i}}_{B}(\varphi_{E})/{D_{si}}_{B}(\varphi_{E})\simeq R^{h}.$$ On the other hand, by Theorem \ref{rank_deRham} we have $D_{A}(\varphi_{E})/{D_{si}}_{A}(\varphi_{E})\simeq R^{r}$. Therefore the rank of the $R$-module $\mathrm{Ext}_{A}(E,\hat{\mathbb{G}}_{a})=D_{A}(\varphi_{E})/{D_{i}}_{A}(\varphi_{E})$ is $r-h$ and we are done.
\end{proof}

\color{black}

\section{ Finite generation of ${\bX}_{\infty}(E)$ as an $R\{\phi^{*}\}$-module}
\label{fingenX}
In this section, given any abelian Anderson $A$-module $E$, we will define the $R$-modules $\bH(E)$ and $\bXp(E)$ and will
show that they are free and finitely generated over $R$. As a consequence, 
we will prove that $\bX_\infty(E)$ is a finite and free module over 
$R\{\phi^*\}$.

We need some preparation.
Given any $R$-module $M$ we denote  the $K$-vector space $M\otimes _{R}K$  by 
$M_{K}$.
Let $(E,\varphi_{E})$ be an abelian Anderson $A$-module of dimension $d$ and rank $r$. For any $R$-module $M$, let us fix the notation
	$$
	M_{\phi} = R\otimes_{\phi,R}M.
	$$
Then the semilinear map $\phi^*:{\bX}_{n-1}(E) \rightarrow {\bX}_n(E)$ induces a linear map
${\bX}_{n-1}(E)_{\phi}\rightarrow {\bX}_n(E)$, which we will still denote by $\phi^*$. Recall the following short exact sequence of $A$-module $\pi$-formal 
group schemes
$$0\longrightarrow N^{n} \stk{i}{\longrightarrow} J^{n}E \longrightarrow E \longrightarrow 0.$$
Applying $\mathrm{Hom}_{A}(-,\hat{\mathbb{G}}_{a})$ to the above we get,
$$0\longrightarrow \bX(E)\longrightarrow {\bX}_{n}(E) \longrightarrow {\bX}(N^{n})\xrightarrow{\partial} \mathrm{Ext}_{A}(E,\hat{\mathbb{G}}_{a}), $$
where $\partial$ is the connecting $R$-linear map.
Define the image  $\mathbf{I}_{n}(E):=\partial ({\bX}(N^{n}))$ and let
${\mathbf{I}}(E)= \varinjlim {\mathbf{I}}_n(E)$. Note that  
${\bX}(N^{n})$, $\mathrm{Ext}_{A}(E,\hat{\mathbb{G}}_{a})$ are  finite and
free $R$-modules. Hence ${\bX}_{n}(E), ~\mathbf{I}_{n}(E)$ and ${\mathbf{I}}(E)$ 
 are also finite and free since they are submodules of finite and free $R$-modules and $R$ is a discrete valuation ring. Let $h_{n}=\mathrm{rk}(\mathbf{I}_{n}(E))-\mathrm{rk} (\mathbf{I}_{n-1}(E))$.

We say $m_{l}$ is the \emph{lower splitting} number of $E$ if ${\bX}_{m_{l}}(E)/\bX(E)\neq 0$ but ${\bX}_{m_{l}-1}(E)/\bX(E)= 0$. Similarly, $m_{u}$ is said to be the 
\emph{upper splitting} number if $h_{i}=0$ for all $i\geq m_{u}$ and $h_{m_{u}-1}\neq 0$.
From Theorem \ref{rank_Ext}, we know that  $\mathrm{rk}_R \mathrm{Ext}_{A}(E,\hat{\mathbb{G}}_{a})=r-h$. Hence the lower and splitting numbers exist and in fact $m_{l}\leq s$.\medskip

 A $\d$-character $\Theta \in {\bX}_{n}(E)_{K} $ is called \emph{primitive} if $$\Theta \notin \bX(E)_{K}+ u^{*}{\bX}_{n}(E)_{K}+\phi^{*}({\bX}_{n}(E)_{K})_{\phi}.$$
For $i\geq 1$, let $\mathbb{B}_{i}\subset {\bX}_{i}(E)_{K}$ be a finite subset such that the image of its elements in $$\dfrac{{\bX}_{i}(E)_{K}}{u^{*}{\bX}_{i-1}(E)_{K}+\phi^{*}({\bX}_{i-1}(E)_{K})_{\phi}}$$ are distinct and form a $K$-basis. Such a $\mathbb{B}_{i}$ is called a \emph{primitive basis} of ${\bX}_{i}(E)_{K}$. Let $l_{i}$ be the cardinality of $\mathbb{B}_{i}$. 
 Then using Theorem \ref{frob_bij} and an argument similar to Section $7$ of \cite{BS a}, we can show that  
$$ S_{n}(\mathbb{B}_{1})\cup S_{n}(\mathbb{B}_{2})\cup \ldots \cup S_{n}(\mathbb{B}_{n}) ,$$
where $S_{n}(\mathbb{B}_{i})=\{{\phi^{*}}^{h}\Theta ~|~ \Theta\in \mathbb{B}_{i}, 0\leq h\leq n-i \} $ forms a $K$-basis of $\bX_n(E)_K/\bX(E)_{K}$.
Hence we have $$\mathrm{rk}_K {\bX}_{n}(E)_{K}/\bX(E)_{K}=l_{n}+2l_{n-1}+\ldots + n_{l_{1}}.$$
Consequently,
$$\mathrm{rk}_R {\bX}_{n}(E)/\bX(E)=l_{n}+2l_{n-1}+\ldots + n_{l_{1}}.$$

\begin{lemma}\label{upper_splitting} For all $i\geq 2$, we have  $l_{i}=h_{i-1}-h_{i}$. Hence $\{h_{i}\}$ is a weakly decreasing sequence.
\end{lemma}
\proof From the exact sequence,  $$0 \longrightarrow {\bX}_{n}(E)/\bX(E)\longrightarrow {\bX}(N^{n})\xrightarrow{\partial} \mathbf{I}_{n}(E) \longrightarrow 0 ,$$
we have $\mathrm{rk}_R (\mathbf{I}_{n}(E)) = nh- \mathrm{rk}_R ({\bX}_{n}(E)/\bX(E))$. Therefore $h_{n}= nh- \mathrm{rk}_R ({\bX}_{n}(E)/\bX(E)))- ((n-1)h- \mathrm{rk}_R ({\bX}_{n-1}(E)/\bX(E))= h-(l_{1}+l_{2}+\ldots +l_{n})$. Hence $ h_{n-1} - h_{n}= l_{n}$.
\qed

\begin{theorem}\label{fg_X_inf_K}
For any abelian Anderson $A$-module $E$ of rank $r$ and dimension $d$, ${\bX}_{\infty}(E)_{K}$  is  freely generated  
by $h$ $\d$-characters of order at most $r-h+1$ as a 
$K\{\phi^*\}$-module.
\end{theorem}

\proof We know that ${\bX}_{\infty}(E)_{K}/\bX(E)_{K}$ is freely generated by the primitive characters $\mathbb{B}_{1}\cup \mathbb{B}_{2}\cup \ldots \cup \mathbb{B}_{m_{u}}$ as a $K\{\phi^{*}\}$ module. Therefore we have $$l_{m_{l}}+l_{m_{l}+1}+\ldots +l_{m_{u}} = l_{m_{l}}+(h_{m_{l}} - h_{m_{l}+1})+ \ldots + (h_{m_{u}-1} - h_{m_{u}})=l_{{m_{l}}}+h_{{m_{l}}} = h.$$ 
Since $\sum h_{i} \leq r-h $, and $h_{i}$ are weakly decreasing sequence of non-negative integers, we must have $h_{i}=0$ for $i\geq r-h+1$. Therefore $m_{u}\leq r-h+1$. \qed \medskip

Now we define the module of primitive characters as $${\bX}_{\mathrm{prim}}(E):=\underrightarrow{\lim} \dfrac{{\bX}_{n}(E)}{\bX(E)+\phi^{*}({\bX}_{n-1}(E))_{\phi}}.$$ 

\begin{corollary}\label{rk_X_prim} For any abelian Anderson $A$-module $E$ of rank $r$ and dimension $d$, we have $${\bX}_{\mathrm{prim}}(E)_{K}\simeq \dfrac{{\bX}_{m_{u}}(E)_{K}}{\bX(E)_{K}+\phi^{*}(({\bX}_{m_{u}-1}(E))_{K})_{\phi}}.$$ Moreover $\mathbb{B}_{i_{1}}\cup \mathbb{B}_{i_{2}}\ldots \mathbb{B}_{i_{n}} $ is a $K$-basis for ${\bX}_{\mathrm{prim}}(E)_{K}$, where $m_{l}= i_{1}< i_{2}< \ldots < i_{n}= m_{u}$  are the orders of all the primitive characters of $E$. Therefore $\mathrm{rk}_{K}\left({\bX}_{\mathrm{prim}}(E)_{K}\right)=h$.
\end{corollary}
\proof Follows similarly as Corollary 7.7 in \cite{BS a}. \qed
\begin{theorem}\label{phi*_bij}
Let $E$ be any abelian Anderson $A$-module of rank $r$ and dimension $d$. For $n\geq m_{u}$, we have the following commutative diagram
 $${
\xymatrix{
{\bX}_{n}(E)/{\bX}_{n-1}(E)\ar@{^{(}->}[d]_{i^{*}} \ar@{^{(}->}[r]^{\phi^{*}}_{\sim} &{\bX}_{n+1}(E)/{\bX}_{n}(E)\ar@{^{(}->}[d]^{i^{*}}  \\
{\bX}(N^{n})/{\bX}(N^{n-1}) \ar[r]^{\mathfrak{f}^{*}}_{\sim} &{\bX}(N^{n+1})/{\bX}(N^{n}) }}
$$
with all the maps bijective . 
\end{theorem}
\proof From Theorem \ref{frob_bij} we already have $i^{*}$ and $\phi^{*}$ are injective and $\mathfrak{f}^{*}$ is bijective. Consider the exact sequence of finite free $R$-modules
$$0 \longrightarrow {\bX}_{n}(E)/\bX(E)\longrightarrow {\bX}(N^{n})\xrightarrow{\partial} \mathbf{I}_{n}(E) \longrightarrow 0.$$
Since $\mathbf{I}_{n}(E)$ is free, there is a section $s_{n}:\mathbf{I}_{n}(E)\longrightarrow {\bX}(N^{n})$. Moreover, this section implies that the following exact sequence of $R$-modules
$${
\xymatrix{
0\ar[r]&{\bX}_{n}(E)/{\bX}_{n-1}(E) \ar[r]^{i^{*}} &{\bX}(N^{n})/{\bX}(N^{n-1})\ar[r]^{\partial} &\mathbf{I}_{n}E/\mathbf{I}_{n-1}E\ar@/^1pc/[l]^{s}\ar[r]& 0 \\
}}
$$
splits. Similar fact as in Proposition 5.2 in \cite{BS a} implies that ${\bX}(N^{n})/{\bX}(N^{n-1})$ is a free module of rank $d$. Hence $\mathbf{I}_{n}(E)/
\mathbf{I}_{n-1}(E)$ is also free being a sub-module of a free module over a discrete valuation ring. On the other hand,  $$\mathrm{rk}\left(\mathbf{I}_{n}(E)
/\mathbf{I}_{n-1}(E)\right)=\mathrm{rk}(\mathbf{I}_{n}(E))-\mathrm{rk} (\mathbf{I}_{n-1}(E))=h_{n}.$$
Since $h_{n}=0$ for all $n\geq m_{u}$, we have $\mathbf{I}_{n}(E)/\mathbf{I}_{n-1}(E)=0$ and $i^{*}$ is bijective. Since $i^*$ and $\fra^*$ are bijective we 
have that $\phi^{*}$ also bijective. 
\qed

For each $i$, we define $\bb{D}_i$ be a finite subset of elements in ${\bX}_{i}(E)$ such that 
their images in 
$$\dfrac{{\bX}_{i}(E)}{u^{*}{\bX}_{i-1}(E)+\phi^{*}({\bX}_{i-1}(E))_{\phi}}$$ 
form a minimal generating set.
Since $R$ is a discrete valuation ring and the modules are finite over $R$, the 
cardinality of such a minimal generating set is well-defined.
We will call such a $\mathbb{D}_{i}$ a \emph{primitive generating set} of ${\bX}_{i}(E)$. 
Let us define the subset $$S_{n}(\mathbb{D}_{i})=\{{\phi^{*}}^{h}\Theta ~|~\Theta\in \mathbb{D}_{i}, 0\leq h\leq n-i \} .$$
\begin{lemma} 
\label{SnR}
The  $R$-module $\bX_{n}(E)/\bX(E)$ can be generated by
$$ S_{n}(\mathbb{D}_{1})\cup S_{n}(\mathbb{D}_{2})\cup \ldots \cup S_{n}(\mathbb{D}_{n}) .$$
\end{lemma}
\proof We will proceed by induction on $n$. For $n=1$ it 
is clear and the definition of $\bb{D}_1$. 
Let us denote $$ \mathbb{S}_{i}=S_{i}(\mathbb{D}_{1})\cup S_{i}(\mathbb{D}_{2})\cup \ldots \cup S_{i}(\mathbb{D}_{i}).$$ We assume that $\mathbb{S}_{n-1}$ generates $\bX_{n-1}(E)/\bX(E)$ as an $R$-module. Consider the short exact sequence of  $R$-modules
$$0\longrightarrow \dfrac{u^{*}{\bX}_{n-1}(E)}{\bX(E)}+\dfrac{\phi^{*}(
{\bX}_{n-1}(E)}{\bX(E))_{\phi}}\longrightarrow\dfrac{{\bX}_{n}(E)}{\bX(E)}\longrightarrow\dfrac{{\bX}_{n}(E)}{(u^{*}{\bX}_{n-1}(E)+\phi^{*}(
{\bX}_{n-1}(E))_{\phi})}\longrightarrow 0.$$
By the induction hypothesis, $u^*\bX_{n-1}(E)/\bX(E) + \phi^*(\bX_{n-1}(E)/\bX(E))_\phi$
is generated by $\bb{S}_{n-1} \cup \phi^*\bb{S}_{n-1}$. And the quotient 
$R$-module $\dfrac{\bX_n(E)}{(u^*\bX_{n-1}(E)+ \phi^*(\bX_{n-1}(E))_\phi)}$
is generated by the image of $\bb{D}_n$ by definition. Hence $\bX_n(E)/\bX(E)$ is 
generated by 
$$\mathbb{D}_{n}\cup \mathbb S_{n-1}\cup \phi^{*}(\mathbb{S}_{n-1})\subset \mathbb{S}_{n}.$$
and we are done.
\qed


\begin{theorem}\label{X_prim} 
The $R$-module $\bXp(E)$ is free of rank $h$.
\end{theorem}
 \proof Since $R$ is a discrete valuation ring, it is enough to prove that ${\bX}_{\mathrm{prim}}(E)$ is finitely generated and $\pi$-torsion free.
We will now show ${\bX}_{\mathrm{prim}}(E)$ is a finitely generated $R$-module.
By Theorem \ref{phi*_bij}, for $n\geq m_{u}$, we have the following commutative diagram
 $${
\xymatrix{
{\bX}_{n}(E)/{\bX}_{n-1}(E)\ar[d]_{p_{n}} \ar[r]_{\phi^{*}}^{\sim} &{\bX}_{n+1}(E)/{\bX}_{n}(E)\ar[d]^{p_{n+1}}  \\
\dfrac{{\bX}_{n}(E)}{u^{*}{\bX}_{n-1}(E)+\phi^{*}({\bX}_{n-1}(E))_{\phi}} \ar[r]^{\bar{\phi}^{*}} &\dfrac{{\bX}_{n+1}(E)}{u^{*}{\bX}_{n}(E)+\phi^{*}({\bX}_{n}(E))_{\phi}}}}
$$
Since, the vertical map are surjective and $\phi^{*}$ is bijective, we have $\bar{\phi}^{*}$ also surjective. But note that $\bar{\phi}^{*}$ is a zero map. Hence this implies $\dfrac{{\bX}_{n+1}(E)}{u^{*}{\bX}_{n}(E)+\phi^{*}({\bX}_{n}(E))_{\phi}}$ is zero for $n\geq m_{u}$. 
Therefore from Lemma \ref{SnR} it follows that, $\bX_{\infty}(E)$ is generated by $$\mathbb{D}_{1}\cup \mathbb{D}_{2}\cup \ldots \cup \mathbb{D}_{m_{u}}$$ 
as an $R\{\phi^{*}\}$-module. Hence their images generate ${\bX}_{\mathrm{prim}}(E)$ as an $R$-module.

Now we will show that ${\bX}_{\mathrm{prim}}(E)$ is $\pi$-torsion free.
Let $\pi [f] =0$ for some $[f] \in \bX_{\mathrm{prim}}(E)$. This means there exists a $g=\displaystyle{\sum_{I}}a_{I}\mathbf{x}^{I} \in \bX_{\infty}(E)$ such that 
\begin{equation*}
\pi f(\mathbf{x})=\phi^{*}g (\mathbf{x})= g(\phi(\mathbf{x}))=\displaystyle{\sum_{I}}a_{I}\mathbf{x}^{qI}+\pi h(\mathbf{x}).
\end{equation*}
The above equation shows that each $a_{I}$ are divisible by $\pi$. Therefore 
$\tilde{g}(\bx):=\frac{g(\mathbf{x})}{\pi}\in \bX_{\infty}(E)$. Since $\bX_{\infty}(E)$ is free and hence $\pi$-torsion free, we obtain $f=\phi^{*}\tilde{g}$,
that is $[f] =0$. 
Hence ${\bX}_{\mathrm{prim}}(E)$ is free and by Corollary \ref{rk_X_prim} the rank has to be $h$.\qed

\begin{corollary}
\label{FinGenX}
For any abelian Anderson $A$-module of rank $r$ and dimension $d$, ${\bX}_{\infty}(E)$  is  freely generated  
by $h$ $\d$-characters of order at most $r-h+1$ as an 
$R\{\phi^*\}$-module.
\end{corollary}
\proof Since $ \mathbb{D}_{1}\cup\mathbb{D}_{2}\cup \ldots \cup \mathbb{D}_{m_{u}}$ generates ${\bX}_{\mathrm{prim}}(E)$ as an $R$-module, we can choose a minimal generating subset and its cardinality has to be $h$ because $R$ is a 
discrete valuation ring and by Theorem \ref{X_prim}, ${\bX}_{\mathrm{prim}}(E)$ has rank $h$.
Hence $\bX_{\infty}(E)$ will be freely generated by $h$ $\d$-characters of order upto $m_{u}\leq r-h+1$ as in Theorem \ref{fg_X_inf_K}.
\qed
\section{Finiteness of the $R$-module ${\bH}(E)$}\label{ziso}

Here we recall some of the constructions done in \cite{BS b} and adapt them
to our setting. 
Recall the exact sequence $$0\longrightarrow N^{n}\longrightarrow J^{n}E\longrightarrow E\longrightarrow 0.$$ The Teichmuller section $v:E\longrightarrow J^{n}E$ induces an splitting in the Lie algebras
\begin{align*}
 \xymatrix{
0  \ar[r] &\mathrm{Lie} (N^{n})\ar[r] &\mathrm{Lie} (J^{n}E) \ar[r] \ar@/^1pc/[l]^{s_{\mathrm{Witt}}} &\mathrm{Lie} (E)  \ar[r]&0}
\end{align*}
where $s_{\mathrm{Witt}}(\mathbf{x}_{0},\ldots,\mathbf{x}_{n})=(\mathbf{x}_{1},\ldots \mathbf{x}_{n})$ is the corresponding retraction. Given any $\Psi \in {\bX}(N^{n})$ we have the following push out diagram from \eqref{pushout}
\begin{align*}
\xymatrix{
0  \ar[r] &N^{n} \ar[d]^{\Psi} \ar[r] &J^{n}E \ar[d]^{e_{\Psi}} \ar[r] &E \ar@{=}[d] \ar[r]&0 \\
0 \ar[r] &\hat{\mathbb{G}}_{a} \ar[r] &E_{\Psi}^{*} \ar[r] & E \ar[r] &0}
\end{align*}
where $E_{\Psi}^{*}=(J^{n}E\times \hat{\mathbb{G}}_{a})/\Gamma(N^{n})$ and $\Gamma(N^{n})=\{i(\mathbf{z}),-\Psi(\mathbf{z})~|~\mathbf{z}\in N^{n}\}\subset J^{n}E\times \hat{\mathbb{G}}_{a}.$
Let $s_{\Psi}$ denote the induced splitting of Lie algebras of the 
push out extension
\begin{align*}
 \xymatrix{
0  \ar[r] &\mathrm{Lie} (\hat{\mathbb{G}}_{a})\ar[r] &\mathrm{Lie} (E_{\Psi}^{*}) \ar[r]\ar@/^1pc/[l]^{s_{\Psi}} &\mathrm{Lie} (E)  \ar[r]&0}.
\end{align*}
Define $\tilde{s}_{\Psi}:\mathrm{Lie}(J^{n}E)\times \mathrm{Lie}\hat{\mathbb({G}_{a}})\longrightarrow \mathrm{Lie}(\hat{\mathbb{G}_{a}})$ as
\begin{align*}
\tilde{s}_{\Psi}(\mathbf{x},y)=D\Psi(s_{\mathrm{Witt}}(\mathbf{x}))+y 
\end{align*}
and consequently we obtain $\tilde{s}_{\Psi}|_{\mathrm{Lie}(\Gamma(N^{n}))}=0$. Therefore 
\begin{align}\label{s_Witt_def}
s_{\Psi}[\mathbf{x},y]=D\Psi(s_{\mathrm{Witt}}(\mathbf{x}))+y 
\end{align}
is a well defined retraction.
Thus we get the following commutative diagram
\begin{equation}
\xymatrix{
		0 \ar[r] & {\bX}_{n}(E)/\bX(E) \ar[r] \ar[d]^{\tilde{\Upsilon}} & 
		{\bX}(N^{n}) \ar[r] \ar[d]^{\Psi\mapsto (E_{\Psi}^{*},s_{\Psi})}& 
		\mathrm{Ext}_{A}(E,\hat{\mathbb{G}}_{a})  \ar@{=}[d]& \\
		0 \ar[r] &\mathrm{Lie} (E)^{*} \ar[r] & \mathrm{Ext}^{\#}_{A}(E,\hat{\mathbb{G}}_{a}) \ar[r] & 
\mathrm{Ext}_{A}(E,\hat{\mathbb{G}}_{a}) \ar[r] & 0.
	}
\end{equation}\label{Char_to_Hdg_fil_1}
\noindent Let $\mathbf{x}=\bx_0 = (x^{1}, x^{2}, \dots, x^{d})$ be a coordinate function of $E$ centered around the identity section and $\left(\mathbf{x}_{0},\mathbf{x}_{1},\ldots, \mathbf{x}_{n}\right)$ be the induced
 coordinate functions on $J^{n}E$, where $\mathbf{x}_{i}=(x^{1}_{i}, x^{2}_{i} 
\dots x^{d}_{i})$ are the $i$th Witt coordinates associated to
$\bx$.
 Therefore given any $\Theta \in {\bX}_{n}(E)$, we have 
\begin{equation}
D\Theta=(P_{0},\ldots,P_{n}),\quad \mathrm{where}~
P_{i}=\begin{pmatrix}
           {\partial \Theta}/{\partial x^{1}_{i}} \\
           {\partial \Theta}/{\partial x^{2}_{i}} \\
           \vdots \\
          {\partial \Theta}/{\partial x^{h}_{i}}
           \end{pmatrix}.
\end{equation}\label{Derivative_matrix}
The following proposition describes the map $\tilde{\Upsilon}:{\bX}_{n}(E)\longrightarrow\mathrm{Lie}(E)^{*}$ in diagram \eqref{Char_to_Hdg_fil_1} explicitly.

\begin{proposition} Let $\Theta$ be a character in ${\bX}_{n}(E)$ and $\Psi=i^{*}\Theta \in {\bX}(N^{n})$,
\begin{itemize}
\item[(a)] The map $\tilde{\Upsilon}:{\bX}_{n}(E)\longrightarrow \mathrm{Lie}(E)^{*}$ in diagram \eqref{Char_to_Hdg_fil_1} is given by $$\Theta \mapsto -D(\Theta \circ v).$$
\item[(b)] Let $\bar{\Theta}=\phi^{*}\Theta $, then $-D(\bar{\Theta}\circ v)=0$.
\item[(c)]If $\Psi\in i^{*}\phi^{*}({\bX}_{n}(E))$, then the class $(E_{\Psi}^{*},s_{\Psi})$ is trivial in $\mathrm{Ext}^{\#}_{A}(E,\hat{\mathbb{G}}_{a})$.
\end{itemize}
\end{proposition}\label{diagram}

\proof  Note that for any $\Theta \in {\bX}_{n}(E)$, the corresponding extension $E_{\Theta}^{*}$ is trivial and we have the following isomorphism 
$$\dfrac{J^{n}E\times \hat{\mathbb{G}}_{a}}{\Gamma(N^{n})}=E_{\Theta}^{*}\simeq E \times \hat{\mathbb{G}}_{a} $$ given by $[\mathbf{x},y]\mapsto (u(\mathbf{x}),\Theta(\mathbf{x})+y)$. By \eqref{s_Witt_def} the inverse is given by $$(\mathbf{x}_{0},y)\mapsto [v(\mathbf{x}_{0}),y-\Theta(v(\mathbf{x}_{0}))].$$ 
Hence the composition $E\rightarrow E \times \hat{\mathbb{G}}_{a} \rightarrow E_{\Theta}^{*}\rightarrow \hat{\mathbb{G}}_{a}$ is given by $\mathbf{x}_{0}\mapsto -\Theta\circ v(\mathbf{x}_{0})$ and passing to the induced map on Lie algebras, we obtain $\tilde{\Upsilon}(\Theta)=-D(\Theta\circ v)$ which proves (a).\medskip

Recall from Section 4B in \cite{BS b} that the Frobenius morphism $\phi:
J^{n+1}E \map J^nE$ is given by $\phi(\bx_{i})=\bx_{i}^{\hq}+\pi \bx_{i+1}$ 
in Buium-Joyal coordinates which are same as Witt coordinate for $i=0,1$ by Proposition 3.2 in \cite{BS b}. Thus we have $$\bar{\Theta}(v(\mathbf{x}_{0}))=\Theta(\phi(v(\mathbf{x}_{0})))=\Theta(\phi(\mathbf{x}_{0},0,\ldots,0))=\Theta(\mathbf{x}_{0}^{(\hq)},0,\ldots,0)=\Theta(v(\mathbf{x}_{0}^{(\hq)})).$$
Hence $D(\bar{\Theta}(v(\mathbf{x}_{0}))=0$ which proves (b). Then (c) follows from (b) by the commutativity of the left rectangle of diagram \eqref{Char_to_Hdg_fil_1}.
\qed

 We define the $R$-module $${\bH}_{n}(E):=\dfrac{{\bX}(N^{n})}{i^{*}\phi^{*}({\bX}_{n-1}(E)_{\phi})}.$$
Then $u:N^{n+1}\longrightarrow N^{n}$ induces a map $u^{*}:{\bX}(N^{n})\longrightarrow {\bX}(N^{n+1})$. Also $u^{*}$ commutes with $i^{*}$ and $\phi^{*}$. Hence we have a map $$u^{*}:{\bH}_{n}(E)\longrightarrow {\bH}_{n+1}(E).$$
Define $ {\bH}(E):=\underrightarrow{\lim}$ ${\bH}_{n}(E) $ and $ \mathbf{I}(E):=\underrightarrow{\lim}$ $\mathbf{I}_{n}(E) $.
Similarly, $\mathfrak{f}:N^{n+1}\longrightarrow N^{n}$ induces a linear map $\mathfrak{f}^{*}:{\bX}(N^{n})\longrightarrow {\bX}(N^{n+1})$ which descends to a semilinear map of $R$-modules $$\mathfrak{f}^{*}:{\bH}_{n}(E)\longrightarrow {\bH}_{n+1}(E)$$ by Proposition \eqref{Fr identity} (ii). This induces a semilinear endomorphism $\mathfrak{f}^{*}:{\bH}(E)\longrightarrow {\bH}(E)$. By Proposition \ref{diagram} and  diagram \eqref{Char_to_Hdg_fil_1} we have,
$${
\xymatrix{
0  \ar[r] &\dfrac{{\bX}_{n}(E)}{\bX(E)+\phi^{*}({\bX}_{n-1}(E))_{\phi}}\ar[d]^\Upsilon \ar[r] &{\bH}_{n}(E)\ar[d]^\Phi \ar[r] &\mathbf{I}_{n}(E)\ar@{^{(}->}[d] \ar[r]&0 \\
0 \ar[r] &\mathrm{Lie} (E)^{*} \ar[r] &\mathrm{Ext}^{\#}_{A}(E,\hat{\mathbb{G}}_{a}) \ar[r] &\mathrm{Ext}_{A}(E,\hat{\mathbb{G}}_{a}) \ar[r] &0.}
}$$
Therefore passing to the limit and by Definition \ref{deRhamdef} we get,
\label{map to_deRham}
\begin{equation} 
{\xymatrix{
0  \ar[r] &{\bX}_{\mathrm{prim}}(E)\ar[d]^\Upsilon \ar[r] &{\bH}(E)\ar[d]^\Phi \ar[r] &\mathbf{I}(E)\ar@{^{(}->}[d] \ar[r]&0 \\
0 \ar[r] &\mathrm{Lie} (E)^{*} \ar[r] &\bH^{*}_{\mathrm{dR}}({E}) \ar[r] &\mathrm{Ext}_{A}(E,\hat{\mathbb{G}}_{a}) \ar[r] &0.}
}.
\end{equation}
To describe $\Upsilon$ explicitly, let  $\Theta_{1},\ldots,\Theta_{h}$ be a basis of ${\bX}_{\mathrm{prim}}(E)_{K}$. For each $\Theta_{j}\in {\bX}_{n}(E)$ the derivative  matrix is given by  $D\Theta_{j}=(P_{0j},P_{1j},\ldots, P_{nj})$ as in \eqref{Derivative_matrix}. We define the following matrix 
\begin{equation*}
P=(P_{01},P_{02},\ldots,P_{0h})=\begin{bmatrix}
{\partial \Theta_{1}}/{\partial x^{1}_{0}} &  \dots &  {\partial \Theta_{h}}/{\partial x^{1}_{0}} \\

    \vdots & \dots & \vdots \\
{\partial \Theta_{1}}/{\partial x^{h}_{0}} & \dots & {\partial \Theta_{h}}/{\partial x^{h}_{0}}  
\end{bmatrix}.
\end{equation*}

\begin{theorem}\label{finite_dim_thm} 
For an Anderson module $E$, ${\bH}(E)$ is a free $R$-module of rank 
$\leq r$. Moreover, the corresponding map $\Phi_{K}$ is injective if and only if 
the $h\times h$ matrix $P$ is invertible over $K$.
\end{theorem}
\proof By Theorem \ref{X_prim} we have $\bX_{\mathrm{prim}}(E)$ is a free 
$R$-module of rank $d$.
From the short exact sequence in equation \eqref{map to_deRham} it follows that ${\bH}(E)$ is free since $\bXp(E)$ and $\bI(E)$ are both free. 
Moreover, we have 
$$\mathrm{rk}({\bH}(E))= \mathrm{rk}({\bX}_{\mathrm{prim}}(E)) + \mathrm{rk}(\mathbf{I}(E))\leq h+(r-h)=r.$$ 
By Proposition \ref{diagram}(a) we have $\Upsilon(\Theta_{i})=-D\Theta_{i}\circ v=-P_{0i}$. Thus the matrix of $\Upsilon$ with respect to the basis $\Theta_{1},\ldots,\Theta_{h}$ is equal to $-P$. Since the map $\mathbf{I}(E)_{K}\rightarrow \mathrm{Ext}_{A}(E,\hat{\mathbb{G}}_{a})_{K} $ is injective, note that $\Phi_{K}$ is injective if and only if $\Upsilon$ is injective. Also from Corollary \ref{rk_X_prim}, we have $\mathrm{rk}_{K}({\bX}_{\mathrm{prim}}(E)_{K})=\mathrm{rk}_{K}(\mathrm{Lie} (E)^{*}_{K})=h$. Therefore $\Phi_{K}$ is injective if and only if $P$ is invertible. \qed

\section{Canonical Lift Criteria of Drinfeld modules}\label{Drinfeld_module}

Let $(E,\varphi_{E})$ be a Drinfeld $A$-module of rank $r$ over $R$ where 
$A$ is taken to be as in Section \ref{AndMod}. Then $\vp$ satisfies
\begin{align}
\label{struct-coeff}
\varphi_{E}(z)=\displaystyle{\sum_{i=0}^{r}}a_{i}\tau^{i},
\end{align} 
where the structure coefficients $a_i \in R$ and $a_r$ is invertible.
Then recall that for each $n\geq 1$ we have the canonical exact sequence 
$$0\longrightarrow N^{n}\longrightarrow J^{n}E\longrightarrow E\longrightarrow 0.$$
We will compute the $A$-action on $N^{1}$ below. Consider the following commutative  diagram

$${
\xymatrix{
W_{1}(B)\ar[d]_{\varphi_{J^{1}E}(z)} \ar[r]^{w} &B\times B^{\phi}\ar[d]^{\varphi_{E}(z)\times \varphi_{E^{\phi}}(z)}  \\
W_{1}(B) \ar[r]^{w} &B\times B^{\phi} }}
$$
where $w$ is the ghost map given by $w(x_{0},x_{1})=(x_{0},x_{0}^{\hq}+\pi x_{1})$ and $B^{\phi}$ is the $R$-algebra given by the structure map $R\xrightarrow{\phi}R\rightarrow B$.
Let $\varphi_{J^{1}E}(z)(x_{0},x_{1})=(z_{0},z_{1})$. Also we have  $$\left(\varphi_{E}(z)\times \varphi_{E^{\phi}}(z)\right)(b_{0},b_{1})=(\displaystyle{\sum_{i=0}^{r}} a_{i}b_{0}^{q^{i}},\displaystyle{\sum_{i=0}^{r}} \phi(a_{i})b_{1}^{q^{i}}).$$
Then by the commutativity of the diagram we have  $$w(z_{0},z_{1})= \left(\varphi_{E}(z)\times \varphi_{E^{\phi}}(z)\right)(x_{0},x_{0}^{\hq}+\pi x_{1}).$$
Therefore $z_{0}=\displaystyle{\sum_{i=0}^{r}} a_{i}x_{0}^{q^{i}}=\varphi_{E}(z)(x_{0})$ and 

\begin{eqnarray*}
z_{0}^{\hq}+\pi z_{1}&=& \displaystyle{\sum_{i=0}^{r}} \phi(a_{i})(x_{0}^{\hq}+\pi x_{1})^{q^{i}}\\
 \pi z_{1}&=& \displaystyle{\sum_{i=0}^{r}} (\phi(a_{i})-a_{i}^{\hq})x_{0}^{q^{i+f}}+\displaystyle{\sum_{i=0}^{r}} \pi^{q^{i}}\phi(a_{i})x_{1}^{q^{i}}\\
 z_{1}&=&\displaystyle{\sum_{i=0}^{r}} a_{i}'x_{0}^{q^{i+f}}+\displaystyle{\sum_{i=0}^{r}} \pi^{q^{i}-1}\phi(a_{i})x_{1}^{q^{i}}.
\end{eqnarray*}
We obtain the action of $\varphi_{N^{1}}(z)$ by setting $x_{0}=0$. Therefore we have
\begin{equation}\label{action_N^1_Drin}
\varphi_{N^{1}}(z)=\displaystyle{\sum_{i=0}^{r}}\pi^{q^{i}-1}\phi(a_{i})\tau^{i} .  
\end{equation} 
Note that the $A$-action of $J^{1}E$ is given by 
\[
\varphi_{J^{1}E}(z)=
\begin{bmatrix}
   \varphi_{E}(z)  & 0  \\
    \eta_{J^{1}E}(z) & \varphi_{N^{1}}(z)  & 
     
\end{bmatrix},
\]
where 
\begin{equation}\label{bider_1}
\eta_{J^{1}E}(z)=\displaystyle{\sum_{i=0}^{r}} a_{i}'\tau^{i+f}.  
\end{equation}
Recall from \cite{BS b} for a Drinfeld module the upper and lower splitting numbers are equal and we say the splitting number of $E$ is $m$. From Theorem 8.3 in \cite{BS b} we know that the splitting number $m\leq r$ and $m=1$ if and only if $E$ has a canonical lift. 
Also recall the following exact sequence of $R$-modules
\begin{equation}\label{eqn_Ext}
0\longrightarrow {\bX}_{n}(E) \xrightarrow{i^{*}} {\bX}(N^{n})\xrightarrow{\partial} \mathrm{Ext}_{A}(E,\hat{\mathbb{G}}_{a}).
\end{equation}
At $n=m$, we have ${\bX}_{n}(E)\neq 0$ but ${\bX}_{n}(E)=0$ for all $n<m$. Hence there exist a nonzero $\Psi\in {\bX}(N^{m})$ such that $\partial(\Psi)=0$.  We know from Theorem \ref{basis_canonical_char} that $\{\Psi_{1},\Psi_{2},\ldots,\Psi_{m}\}$ is an $R$-basis of ${\bX}(N^{m})$. For any $\Psi \in \bX(N^n)$ 
there exists unique $\tilde{\lambda}_{1},\ldots,\tilde{\lambda}_{m}\in R$ such that $$\Psi=\tilde{\lambda}_{m}\Psi_{m}-\tilde{\lambda}_{m-1}\Psi_{m-1}-\ldots -\tilde{\lambda}_{1}\Psi_{1}. $$
Note that $\tilde{\lambda}_{m}\neq 0$, since otherwise $\Psi\in {\bX}(N^{m-1})={0}.$ Therefore we have a unique $\Theta_{m}\in {\bX}_{m}(E)$ such that
\begin{equation}\label{eqn_lamda}
i^{*}\Theta_{m}=\Psi_{m}-\lambda_{m-1}\Psi_{m-1}-\ldots -\lambda_{1}\Psi_{1},  
\end{equation} 
where $\lambda_{i}=\tilde{\lambda}_{i}/\tilde{\lambda}_{m}$.  
Indeed, from Theorem 9.8 in \cite{BS b} we know that $\lambda_{1},\ldots \lambda_{m-1} \in R$. Also note that for each $n\geq 1$, we have an inclusion 
$${
\xymatrix{
 {\bX}_{n}(E)/{\bX}_{n-1}(E) \ar@{^{(}->}[r] &{\bX}(N^{n})/{\bX}(N^{n-1})\cong R\langle\Psi_{n}\rangle \ \mathrm{by}\ \mathrm{Theorem} \ \ref{basis_canonical_char}.}
}$$
Therefore we have ${\bX}_{m}(E)=R\langle\Theta_{m}\rangle$ .
\medskip
Also from Theorem 9.8 in \cite{BS b}, we know that $\lambda_{1},\ldots \lambda_{m-1} \in R$.
Also note that for each $n\geq 1$, we have an inclusion 
$${
\xymatrix{
 {\bX}_{n}(E)/{\bX}_{n-1}(E) \ar@{^{(}->}[r] &{\bX}(N^{n})/{\bX}(N^{n-1})\cong R\langle\Psi_{n}\rangle \ \mathrm{by}\ \mathrm{Theorem} \ \ref{basis_canonical_char}.}
}$$
Therefore we have ${\bX}_{m}(E)=R\langle\Theta_{m}\rangle$ .
\medskip

Let us recall from Section \ref{splitting_J^n} that for each $n\geq 1$ we have the \textup{Teichm{$\ddot{u}$}ller section} $\teich:E\longrightarrow J^{n}E$  which is only an $\mathbb{F}_{q}$-linear map. Therefore we have a splitting of $J^{n}E$  as $\mathbb{F}_{q}$-module schemes given by $$J^{n}E\simeq E\times N^{n}.$$ Note that in this splitting the \textup{Teichm{$\ddot{u}$}ller section} is given by $\teich(x)=(x,0,\ldots ,0)$. We would like to emphasize that the above splitting is not as $A$-module schemes in general. Given any character $\Theta\in{\bX}_{n}(E)$ consider the following commutative diagram 
$${
\xymatrix{
E \ar[r]^\teich \ar[rd]_{\teich^{*}\Theta} &
J^{n}E \ar[d]^\Theta &
N^{n} \ar[l]_i \ar[ld]^{i^{*}\Theta}\\
&\hat{\mathbb{G}}_a}
}$$
Therefore we can write \begin{equation}\label{split_theta}
\Theta(x_{0},\ldots,x_{n})=g_{\Theta}(x_{0})+i^{*}\Theta(x_{1},\ldots,x_{n}),
\end{equation} 
where $g_{\Theta}=\teich^{*}\Theta$ is an $\mathbb{F}_{q}$-linear map and can be expressed as a restricted power series over $R$. However $i^{*}\Theta$ is $A$-linear since $i$ is the $A$-linear inclusion of $N^{n}$ to $J^{n}E$.
\begin{definition} Let $E$ be a Drinfeld module of rank $r$ with 
splitting number $m$. We define $\gamma(E)=\pi g'_{\Theta_{m}}(0)$, where $\Theta_{m}$ is the unique element in ${\bX}_{m}(E)$ as above. We will denote it by $\gamma$ once $E$ is fixed. 
\end{definition} 

\bigskip

 Let $(E,\varphi_{E})$ be a Drinfeld $A$-module of rank $2$ over $R$. We know from Theorem 8.3 in \cite{BS b}, that $E$ is CL iff  the splitting number $m=1$. Therefore if $E$ is not CL then we have from (\ref{eqn_Ext}),  $\partial \Psi_{1}\neq 0$.
 
\begin{definition} Given a  Drinfeld module $(E,\varphi_{E})$ of rank $2$ over $R$, we define 
\begin{align*}
\lambda_{1}(E)=\left\{\begin{array}{ll}
		0, & \mb{if } E~ \mathrm{is~ CL}\\
\partial(\Psi_{2})/\partial(\Psi_{1}), & \mb{otherwise. }
	\end{array} 
	\right.
\end{align*}

\end{definition}
We would like to mention that if $E$ is not CL then $\lambda_{1}(E)$ is the same as in equation (\ref{eqn_lamda}).

\begin{theorem}
Let $(E,\varphi_{E})$ be a Drinfeld $A$-module of rank $2$ over $R$, then $E$ is $\mathrm{CL}$ if  both $\gamma$ and $\lambda_{1}$ are zero.
\end{theorem}
\proof Suppose $\lambda_1 = \gamma = 0$ but $E$ is not CL. We have $\lambda_{1}$ and $\gamma$ are equal to $0$. Then recall the exact sequence of $K$-vector spaces 
$$ 0  \longrightarrow {\bX}_{\mathrm{prim}}(E)_{K}\longrightarrow {\bH}(E)_{K} \longrightarrow\mathbf{I}(E)_{K} \longrightarrow 0 .$$
Since $\gamma=0$, from Proposition \ref{Fr identity} we have $\mathfrak{f}^{*}i^{*}=i^{*}\phi^{*}$. Therefore for each $n\geq 1$ the following diagram commutes
$${
\xymatrix{
0  \ar[r] &{\bX}_{n-1}(E)_{K}\ar[d]^{\phi^{*}} \ar[r]^{i^{*}} &{\bX}(N^{n-1})_{K}\ar[d]^{\mathfrak{f}^{*}} \ar[r]^{\partial} &\mathbf{I}_{n-1}(E)_{K}\ar[d]^{\bar{\mathfrak{f}}^{*}} \ar[r]&0 \\
0 \ar[r] &{\bX}_{n}(E)_{K} \ar[r]^{i^{*}} &{\bX}(N^{n})_{K} \ar[r]^{\partial} &\mathbf{I}_{n}(E)_{K} \ar[r] &0.}}
$$
Hence the semi-linear operator $\mathfrak{f}^{*}$ on ${\bH}(E)_{K}$ can be restricted to both on   ${\bX}_{\mathrm{prim}}(E)_{K}$ and $\mathbf{I}(E)_{K}$. Note that the restriction of $\mathfrak{f}^{*}$ on ${\bX}_{\mathrm{prim}}(E)_{K}$ is $\phi^{*}$ which is zero. Since $E$ is not CL, $\partial(\Psi_{1})\neq 0$ and forms a $K$-basis of $\mathbf{I}(E)_{K} $. By the above commutative diagram  we have $$\bar{\mathfrak{f}}^{*}\partial(\Psi_{1})=\partial(\mathfrak{f}^{*}(\Psi_{1}))=\partial(\Psi_{2})=\lambda_{1}\partial(\Psi_{1})=0, \ \mathrm{since} \ \lambda_{1}=0 .$$ 
Therefore the semi-linear operator $\mathfrak{f}^{*}$ on ${\bH}(E)_{K}$ is zero as it is zero  both on $\bXp(E)_K$ and $\bI(E)_K$ and the exact sequence of $K$-vector spaces splits. But when $E$ is not CL, by our earlier observation,
 the matrix associated to the semi-linear operator is given by
\[
[\mathfrak{f}^{*}]=
\begin{bmatrix}
   0  & -\gamma  \\
    1 & \phi(\lambda_{1})& 
     
\end{bmatrix}.
\]
This is a contradiction because the above matrix is never zero for any value of $\lambda_{1}$ and $\gamma$.
\qed

\begin{lemma}\label{CL_gamma} 
Let $E$ be a  Drinfeld $A$-module of rank $r$ over $R$. If $E$ admits a Canonical Lift $\psi(x)=x^{\hq}+\pi h(x)$, then $\gamma=-\pi h'(0)$.
\end{lemma}

\proof Suppose $E$ admits the canonical lift $\psi(x)=x^{\hq}+\pi h(x)$.   
Therefore the canonical exact sequence of $A$-module schemes  
$$0\longrightarrow N^{1}\longrightarrow J^{1}E\longrightarrow E\longrightarrow 0$$
splits. The splitting is explicitly given by the section $\mathfrak{s}(x)=(x,h(x)):E\longrightarrow J^{1}E$. The induced retraction from $J^{1}E$ to $N^{1}$ is given by $\mathfrak{r}(x_{0},x_{1})=x_{1}-h(x_{0})$. By Theorem $5.8$ in \cite{BS b} we know that $\nu_{1}:N^{1}\longrightarrow\hat{\mathbb{G}}_{a}$ is an isomorphism.
Hence $\Theta_{1}=\nu_{1}\circ \mathfrak{r}$. Therefore $g_{\Theta_{1}}=v^{*}\Theta_{1}=\nu_{1}\circ(-h)$
Since $E$ is CL we have $m=1$ and $\gamma=\pi g'_{\Theta_{1}}(0)=\pi\left(\nu_{1}\circ(-h)\right)'(0)=-\pi h'(0)$, since the linear term of $\nu_{1}(x)$ is $x$.
\qed

\begin{theorem}
Let $E$ be a Drinfeld module of rank $r$ over $R$. There are restricted power series $f_{i}$ in $2r+1$  variables over $R$ such that $E$ has a canonical lift iff    $$ f_{i}(a_{1},a_{2},\ldots ,a_{r},a_{1}',a_{2}',\ldots ,a_{r}',a_{r}^{-1})=0,\ for\ 1\leq i\leq r-1$$ 
where $a_1,\dots, a_r$ are the structure coefficients of $E$ as in 
(\ref{struct-coeff}).
\end{theorem}

\proof Recall the $A$-linear isomorphism $\nu_{1}:N^{1}\longrightarrow\hat{\mathbb{G}}_{a}$. From Theorem 5.8 in \cite{BS b}, we have 
$$\nu_{1}=\displaystyle{\sum_{i=0}^{\infty}}b_{i}\tau^{i},$$
with $v(b_{i})\geq i$. Indeed we have
\begin{align*}
b_{0}=1 \
\mathrm{and}  \ b_{i}=(\pi-\pi^{q^{i}})^{-1}\displaystyle{\sum_{j=1}^{i}}b_{i-j}c_{j}^{q^{i-j}},
\end{align*} 
where $\varphi_{N^{1}}(z)=\displaystyle{\sum_{j=0}^{r}}c_{j}\tau^{j}$. By equation \eqref{action_N^1_Drin}, we have $c_{j}=\pi^{q^{j}-1}\phi(a_{j})$.
Now consider the following $R$-linear map
 $$ {\bX}(N^{n})\xrightarrow{\partial} \mathrm{Ext}_{A}(E,\hat{\mathbb{G}}_{a}) .$$
From Theorem 8.3 in \cite{BS b}, we know that a Drinfeld module of rank $r$ is CL if and only if
 $$\partial(\Psi_{1})=0 \Leftrightarrow {\bX}_{1}(E)\neq 0 .$$ We know that $\mathrm{Ext}_{A}(E,\hat{\mathbb{G}}_{a})\cong D(\varphi_{E})/D_{i}(\varphi_{E})$ as $R$-modules. We can also identify $D(\varphi_{E})\cong \tau R\{\tau\hat{\}}$ and the inner bi-derivations are given by
$\pi m-m\circ \varphi_{E},$ for $m\in R\{\tau\hat{\}}$. The canonical extension  $J^{1}E$ of $E$ by $N^{1}$ attach a unique $(\varphi_{E},\varphi_{N^{1}})$-bi-derivation $\eta_{J^{1}E}$ given by 
\begin{equation*}
\eta_{J^{1}E}=\pi' \tau^{f} + a_{1}'{\tau}^{1+f}+\ldots+ a_{r}'{\tau}^{r+f}.
\end{equation*}
Therefore we have $\partial(\Psi_{1})=[\nu_{1}\circ \eta_{J^{1}E}]$ in $\mathrm{Ext}_{A}(E,\hat{\mathbb{G}}_{a})$. An easy computation shows that 
\begin{equation}\label{im_psi}
\nu_{1}\circ \eta_{J^{1}E}=\displaystyle{\sum_{k\geq f}\sum_{j=0}^{r}}b_{k-j-f}a_{j}'^{q^{k-j-f}}\tau^{k}.
\end{equation}
Setting $m=\tau^{j}$ we get,
$$\tau^{r+j}= a_{r}^{q^{-j}}((a_{0}-a_{0}^{q^{j}})\tau^{j}-(\displaystyle{\sum_{i=0}^{j-1}}a_{i}^{q^{j}}\tau^{i+j})).$$
 Therefore $\{[\tau],[\tau^{2}],\ldots [\tau^{r-1}] \}$ span $ \mathrm{Ext}_{A}(E,\hat{\mathbb{G}}_{a})$ as an $R$-module. Since $ \mathrm{Ext}_{A}(E,\hat{\mathbb{G}}_{a})$ has rank $r-1$ as $R$-module and $R$ is a discrete valuation ring, the subset $\{[\tau],[\tau^{2}],\ldots [\tau^{r-1}] \}$ form an $R$-basis of $\mathrm{Ext}_{A}(E,\hat{\mathbb{G}}_{a})$.
 For each $n\geq r$  substituting $[\tau^{n}]$ by the above formula in \eqref{im_psi}, we obtain 
$$ \partial(\Psi_{1})=[\nu_{1}\circ \eta_{J^{1}E}]=\displaystyle{\sum_{i=1}^{r-1}} f_{i}(a_{1},a_{2},\ldots ,a_{r},a_{1}',a_{2}',\ldots ,a_{r}',a_{r}^{-1})[\tau^{i}]=0, \ ~\mathrm{in}~ \ \mathrm{Ext}_{A}(E,\hat{\mathbb{G}}_{a}).$$
where $f_{i}$ are restricted power series in $2r+1$  variable because $v(b_{i})\geq i$. Therefore, we have $f_{i}=0$ for each $i=1,\ldots,r-1$. \qed \medskip

\bibliographystyle{amsalpha}

\end{document}